\documentclass[reqno]{amsart}
\usepackage{amsmath,amssymb,amsthm,amsrefs, amscd}
\usepackage[all]{xy}

\usepackage{xcolor}
\usepackage[english]{babel}
\usepackage[utf8]{inputenc}	

\newtheorem{thm}{Theorem}[section]
\newtheorem{pro}[thm]{Proposition}

\newtheorem{lem}[thm]{Lemma}

\newtheorem{cor}[thm]{Corollary}
\theoremstyle{definition}
\newtheorem{defi}[thm]{Definition}

\newtheorem{rem}[thm]{Remark}

\newtheorem{exa}[thm]{Example}

\newcommand{\vu}{\vspace{0.1cm}}

\usepackage{hyperref}
\numberwithin{equation}{section}\theoremstyle{plain}
\theoremstyle{definition}

\allowdisplaybreaks

\title[PARTIAL SMASH COPRODUCT  OF MULTIPLIER HOPF ALGEBRAS]{PARTIAL SMASH COPRODUCT OF MULTIPLIER HOPF ALGEBRAS}
\author[Fonseca, Fontes and Martini]{Graziela Fonseca, Eneilson Fontes and Grasiela Martini}

\address[Fonseca]{Instituto Federal Sul-Rio-Riograndense, Rio Grande do Sul, Brazil}
\email{grazielalangone@gmail.com}

\address[Fontes]{Universidade Federal do Rio Grande, Brazil}
\email{eneilsonfontes@furg.br}
\address[Martini]{Universidade Federal do Rio Grande, Brazil}
\email{grasiela.martini@furg.br}

\begin{document}
	
	\allowdisplaybreaks

\begin{abstract}
In this work we define partial (co)actions on multiplier Hopf algebras, we also present examples and properties. From a partial comodule coalgebra we construct a partial smash coproduct generalizing the constructions made by the L. Delvaux in \cite {Lydia} and E. Batista and J. Vercruysse in \cite {Batista}.
\end{abstract}

%\vu
%{\nd\scriptsize{\bf Keywords:}  Multiplier Hopf algebras, partial comodule coalgebra, partial smash coproduct\\}
%{\nd\scriptsize{\bf Mathematics Subject Classification:} primary 16T99; secondary 20L05}



\maketitle

\tableofcontents

\section{Introduction}

The theory of multiplier Hopf algebras was introduced in 1994 by A. Van Daele, with the aim to generalize Hopf algebras to a nonunital context \cite{Multiplier}. The driving example was an algebra $A_G$ of finitely supported complex functions on a group $G$, \textit{i.e.}, functions that assume nonzero values only for a finite set of elements of $G$. The multiplier algebra $M(A_G)$ consists of all complex functions on $G$ and $A_G \otimes A_G$ can be identified with the complex functions with finite support from $G\times G$ to $\mathbb{C}$. Furthermore, $A_G$ is a multiplier Hopf algebra with the comultiplication given by $\Delta(f)(p,q)=f(pq)$, the counit $\varepsilon(f)=f(1_G)$ and the antipode $(S(f))(p)=f(p^{-1})$, for all $f\in A_G$ and $p,q\in G$.

Some years later, in \cite{Timmermann}, T. Timmermam  reviewed the concept of multiplier Hopf algebras as a particular case of multiplier bialgebras, which extends the notion of bialgebra related to Hopf algebras theory. In this case, a \textit{multiplier bialgebra} is a nondegenerate algebra $A$ equipped with a nondegenerate homomorphism $\Delta: A\longrightarrow M(A\otimes A)$ such that
	
	$i)$ the following subsets of $M(A\otimes A)$ are contained in $A\otimes A\subseteq M(A\otimes A)$:
	
	\begin{center}$\Delta(A)(1\otimes A),\, \,\,\,\,\,\,\,\,\,\Delta(A)(A\otimes 1),\,\,\,\,\,\,\,\,(A\otimes 1)\Delta(A),\,\,\,\,\,\,\,\,(1\otimes A)\Delta(A);$
	\end{center}
	
	$ii)$ $\Delta$ is coassociative: 
	$(\Delta\otimes \imath)\Delta=(\imath\otimes\Delta)\Delta.$\\
	
	If the linear maps $T_1,T_2: A\otimes A \longrightarrow A\otimes A$ given by
	\begin{center}	
		$T_1(a\otimes b)= \Delta(a)(1\otimes b)$\ \ \ and\ \ \ $T_2(a\otimes b)= (a\otimes 1)\Delta(b)$
	\end{center}
	are bijectives, we say the multiplier bialgebra $(A, \Delta)$ is a \textit{multiplier Hopf algebra}.
	
	A multiplier Hopf algebra is called \textit{regular} if the multiplier bialgebra $(A, \sigma \Delta)$ is a multiplier Hopf algebra, where $\sigma$ denotes the canonical flip map. It is easy to check that any Hopf algebra is a multiplier Hopf algebra. Conversely, if $(A, \Delta)$ is a multiplier Hopf algebra and $A$ is unital, $A$ is a Hopf algebra. This proves that the notion of  multiplier Hopf algebra is natural extension of Hopf algebra for nonunital algebras.
Moreover, if $(A,\Delta_A)$ and $(B,\Delta_B)$ are two regular multiplier Hopf algebras, then  $(A\otimes B,\Delta)$ is a regular multiplier Hopf algebra with the product and the coproduct defined in the usual way.

Analogously to Hopf algebras theory, there exist unique linear maps $\varepsilon :A\longrightarrow\Bbbk$ and $S: A\longrightarrow M(A)$, given by
\begin{center}
	$(\varepsilon\otimes \imath)(\Delta(a)(1\otimes b))=ab$ \ \ \ \ \ and \ \ \ \ \ $(\imath\otimes\varepsilon)((a\otimes 1)\Delta(b))=ab,$
\end{center}
\begin{center}
	$m(S\otimes \imath)(\Delta(a)(1\otimes b))=\varepsilon(a)b$ \ \ \ \ \ and \ \ \ \ \ $m(\imath\otimes S)((a\otimes 1)\Delta(b))=\varepsilon(b)a,$
\end{center}
for all $a,b$ in $A$, respectively called \emph{counit} and  \emph{antipode}.

 In the context of regular multiplier Hopf algebras, A. Van Dale  introduced in \cite {Frame} a linear dual structure  $$\hat{A}=\{\varphi(\underline{\hspace{0.2cm}} a) ; a\in A, \ \text{and}\ \varphi\ \text{ is a given left integral} \},$$ which is also a regular multiplier Hopf algebra for any dimension. 
 
Another important property is the existence of (bilateral) local units for a multiplier Hopf algebra $(A, \Delta)$. This means that,  for any finite set of elements $a_1, \ldots , a_n$ of $A$ there exists an element $e\in A$ such that $e a_i = a_i = a_i e$, for all $1 \leq i \leq n$. This is used to justify  that $A^2=A$, which allows one to show that the comultiplication $\Delta$ is a nondegenerate homomorphism (\textit cf.  \cite[Appendix]{Multiplier}). This property also supports the use of the Sweedler's notation in suitable situations (see \cite{Action} and \cite{Sweedler}). Indeed, $\Delta(a)(1\otimes
b)$ can be written as $a_{(1)}\otimes a_{(2)}b$ since
$\Delta(a)(1\otimes b) = \Delta(a)(1\otimes
e)(1\otimes b)= (a_{(1)}\otimes a_{(2)})(1\otimes  b)=a_{(1)}\otimes a_{(2)}b$, where $e$ is a local unit for the element $b$. This notation depends on $b\in A$, but we can use the same local unit $e\in A$ for finite elements of $A$. 

 The study of multiplier Hopf algebra (co)actions on  nondegenerate algebras  was initiated in \cite{Galois}, motivating investigations in diverse directions. Afterwards, L. Delvaux, in \cite{Lydia}, investigated the comodule coalgebra notion involving two multiplier Hopf algebras, and this approach led to the development of a smash coproduct between these algebras, obtaining a new multiplier Hopf algebra.
 
Recently, partial (co)actions of multiplier Hopf algebras on  nondegenerate algebras were developed by the authors in \cite{Grasiela, Fonseca}, generalizing the notions introduced by S. Caenepeel and K. Jassen in \cite{Caenepeel} and also the theory constructed by A. Van Daele and Y. H. Zhang in \cite{Galois}. However, the partial (co)actions theory of multiplier Hopf algebras remains a landscape to be explored. The difficulty is that there are a lot of classical expressions making no sense in this broader context.

This present work intends to generalize some notions and results of \cite{Batista} and \cite{Lydia} to the context of partial comodule coalgebra of multiplier Hopf algebras, using ideas such as extensions and projections. This paper is organized as follows. In Section 2 the basic notions of (bi)module over a vector space and comodule coalgebra in the context of multiplier Hopf algebras are reminded. In Section 3 we introduce the definition of a  partial comodule coalgebra. Several results and examples will be illustrated, such as the necessary and sufficient conditions for a partial comodule coalgebra to be global and also a structure of partial comodule coalgebra is induced from a global one via projections. Finally, in Section 4 we finish by constructing a partial smash coproduct  associated to a partial comodule coalgebra, generalizing the constructions presented in \cite{Batista} and \cite{Lydia}.

%We start this section with some remarks, definition and examples of multiplier hopf algebras. We also define global module and we introduce some notations and properties. For more details, we recommend \cite{Multiplier}, \cite{Frame}, \cite{Action} and \cite{Sweedler}.

\vu

\vu

%Dually, if there exists an injective algebra homomorphism $\rho : R\longrightarrow M(R\otimes A)$ satisfying
%\begin{enumerate}
%	\item[(i)] $\rho(R)(1\otimes A)\subseteq R\otimes A$ and $(1\otimes A)\rho(R)\subseteq R\otimes A$,
%	\item[(ii)] $(\rho\otimes \imath)\rho=(\imath\otimes \Delta)\rho$,
%\end{enumerate}
%then we call $R$ a \textit{right $A$-comodule algebra}.
%%If in addition $(R\otimes 1)\rho(R)\subseteq R\otimes A$, $\rho$ is also called \textit{reduced}.
%\label{defcoacao}

\subsection*{Conventions}
Throughout this paper, all vector spaces and algebras will be considered as having a fixed base field $\Bbbk$.

An algebra $A$ is a \emph{nondegenerate algebra} when it has the property  that $a=0$ if $ab =0$ for all $b \in A$ and $b=0$ if $ab=0$ for all $a \in A$. We will denote the multiplier algebra of $A$ by $M(A)$ which is the usual vector space of all the ordered pairs $(\overline{X},\overline{\overline{X}})$ of linear maps on $A$ such that $ a\overline{X}(b)=\overline{\overline{X}}(a)b$, for all $a,b\in A$. It follows immediately that $\overline{X}(ab)=\overline{X}(a)b$ and $\overline{\overline{X}}(ab)=a\overline{\overline{X}}(b)$, for all $a,b\in A$. The product is given by the rule $(\overline{X},\overline{\overline{X}})(\overline{Y},\overline{\overline{Y}})=(\overline{X}\circ \overline{Y},\overline{\overline{Y}}\circ \overline{\overline{X}})$ and the identity element is given by the pair $1=(\imath,\imath)$ where $\imath$ denotes the identity map of $A$. At appropriate situations, we will use the notation
$\overline{X}(a)=Xa$ and $\overline{\overline{X}}(a)=aX$, for all $a\in A.$

 Moreover, there exists a canonical algebra monomorphism $\jmath:A\to M(A)$ given by $a\mapsto (\overline{X}_a,\overline{\overline{X}}_a)$, where $\overline{X}_a$ (resp., $\overline{\overline{X}}_a$) denotes the left (resp., right) multiplication by $a$, for all $a\in A$. Furthermore, if $A$ is unital then $\jmath$ is an isomorphism.

 The unadorned symbol $\otimes$ will always mean $\otimes_{\Bbbk}$, and $1\otimes 1\otimes a$ will be denoted by $1^2 \otimes a$, for all $a\in A$, analogously we will use this convention for similar cases. Finally, the pair $(A,\Delta)$ (or simply $A$) will always denote a multiplier Hopf algebra and $Y$ a nondegenerate algebra, unless others conditions are required.

\vu

%%%%%%%%%%%%%%%%%%%%%%%%%%%%%%%%%
\section{Basic Definitions}
%%%%%%%%%%%%%%%%%%%%%%%%%%%%%%%%
\subsection{Extended (bi)module} The purpose of this section is to recall the concept of extended bimodules, first presented in \cite{Tools}. For this, initially we will suppose that $Y$ is a vector space and $A$ is a nondegenerate algebra.
%\begin{defi}(\cite{Action}) \label{defmodalgebra}Dizemos que $Y$ é um\textit{ $A$-módulo álgebra à esquerda} se,
%	\begin{enumerate}
%		\item[(i)] $Y$ é um $A$-módulo à esquerda unitário;
%		\item[(ii)] $a\triangleright (xy)=(a_1\triangleright x)(a_2\triangleright y)$, para todos $x,y\in Y$ e $a\in A$.
%	\end{enumerate}
%\end{defi}
%
%Analogamente, definimos um \textit{$A$-módulo álgebra à direita}. 
%
%
%\begin{pro}\textnormal{(\cite{Action})} Se $m\in M(Y)$ e $a\triangleright m=0$, para todo $a\in A$, então $m=0$.
%\end{pro}

\begin{defi}\cite{Action} We call $Y$ a \textit{(global) left $A$-module}, if there exits a linear map
		\begin{eqnarray*}
		\triangleright : A\otimes Y &\longrightarrow &Y\\
		a\otimes y &\longmapsto & a\triangleright y
	\end{eqnarray*}
	satisfying $a\triangleright(b\triangleright y) =ab\triangleright y $, for all $a,b\in A$ and $y\in Y$. 
	In this case, we say that $\triangleright$ is a left action of $ A $ on $ Y $. Furthermore, the left $A$-module $Y$ is called nondegenerate if $a\triangleright y=0$, for all $a\in A$, then $y=0$.
	\end{defi}
 Analogously, we define a nondegenerate \textit{right $A$-module}. If $A$ has unit $1_A$, is natural assume that $1_A \triangleright y =y$, for all $y\in Y$. This notion is extended as follows way.

 \begin{defi} A left $A$-module $Y$ is called \textit{unital} if $A\triangleright Y=Y$.
\end{defi}

Any unital module is nondegenerate but the converse need not be true. Considering a regular multiplier Hopf algebra $A$, the property below is fundamental for the use of the classical Sweedler notation in this background.
 
\begin{rem}\cite{Action} \label{localunitinact} If $Y$ is a unital left $A$-module, then, given $a_1, ..., a_n\in A$ and $y_1, y_2, ..., y_m\in Y$, there exists an element $e\in A$ such that $ea_i=a_i=a_ie$, for all $i\in\{1, ..., n\}$ and $e\triangleright y_j =y_j$, for all $j\in\{1, ..., m\}$.
	\label{obs_unidacao}
\end{rem}

%\begin{pro}\textnormal{\cite{Action}} Let $Y$ be a unital left $A$-module and $y\in Y$. If $a\triangleright y=0$, for all $a\in A$, then $y=0$.
%	\label{acaodegenerada}
%\end{pro}
%
%  In this case, the unital left $A$-module $Y$ is called \textit{nondegenerate}. Similarly, this property is proved for unital right $A$-modules.

\begin{pro} \textnormal{\cite{Action}}\label{extendaction} Let $Y$ be a unital left $A$-module via $\triangleright$. There exists a unique extension to a left $M(A)$-module and $1_{M(A)}\triangleright y =y$, for all $y\in Y.$	
\end{pro}

\begin{exa}\cite{Action} Let $A$ be a multiplier Hopf algebra. Then, $A$ is a unital left $A$-module with action determined by its product.
\end{exa}

We present the next remark to justify the items of the definition of comodule coalgebra. Henceforward, to simplify the notation, we will denote $a\triangleright y$ by $ay$, for all $a\in A$, $y\in Y$, similarly for right modules.

\begin{rem}\label{notacaodepares} Let $A$ be a nondegenerate algebra and let $Y$ be a vector space. Consider $A\otimes Y$ with an $A$-bimodule structure defining left and right actions given by the product on the first factor of tensor product. Then,
	$$ax=a(\sum_{i}a_i\otimes y_i)=\sum_{i}aa_i\otimes y_i=(a\otimes 1)x$$
	and, on the other side, we have that $xa=x(a\otimes 1),$ for all $a\in A$ and $x\in A\otimes Y.$
	
	\begin{defi}\cite{Tools} Let $A\otimes Y$ be a nondegenerate $A$-bimodule. We define $M_0(A\otimes Y)$ to be the vector space of pairs $(\lambda, \rho)$ of linear maps from $A$ to $A\otimes Y$ satisfying
		\begin{eqnarray}\label{compatMo}
		a\lambda(b)=\rho(a)b,
		\end{eqnarray}
		for all $a,b\in A.$
		\end{defi}
		
		We can obtain an immersion from $A\otimes Y$ to $M_0(A\otimes Y)$ by associating to each $x\in A\otimes Y$ two linear maps $\lambda$ and $\rho$ defined by $\lambda(a)=x(a\otimes 1)$ \  \ and \  \ $\rho(a)=(a\otimes 1)x.$
		Moreover, it makes sense to write $z(a\otimes 1)$ for $\lambda(a)$ and $(a\otimes 1)z$ for $\rho(a)$ whenever $z=(\lambda,\rho).$
		Therefore we can extend the actions defined in Remark \ref{notacaodepares} in the following way.
		 
		 \begin{defi} The $A$-bimodule $M_0(A\otimes Y)$ defined by $az=(a\lambda(\_),\rho(\_a))$ and $za=(\lambda(a\_),\rho(\_)a),$
			for all $a\in A$ and $z=(\lambda,\rho)\in M_0(A\otimes Y),$ is called \textit{completion} or \textit{extended bimodule} of the $A$-bimodule $A\otimes Y.$
	\end{defi}
		It is easy to check that the construction presented above also works for all vector space $X$ that has a nondegenerate $A$-bimodule structure, so there is always the completion $M_0(X).$ We note that if $X=A$ and $A$ is seen as an $A$-bimodule via its product, then $M_0(A)=M(A).$ For this reason, it is natural to call an element in $M_0(X)$ of multiplier.
	
	In particular, given an algebra $A$ and the vector space $Y$, with the previous notation, we can consider $A\otimes A\otimes Y$ as an $A\otimes A$-bimodule via left and right product on the first two factors of the tensor product. Therefore we can define its completion $M_0(A\otimes A\otimes Y).$ In this case, we will write the actions of $A\otimes A$ as $(a\otimes a'\otimes 1)z$ and $z(a\otimes a'\otimes 1)$ whenever $a,a'\in A$ and $z$ is an element in $M_0(A\otimes A\otimes Y)$, see \cite{Tools}.	
	
	Suppose $X$ a tensor combination of $A$ and $Y$. When $X$ is an $A$-bimodule or a $Y$-bimodule we will denote by $M_{0,i}^A(X)$ the completion of $X$ whenever the action of $A$ is on the $i$-th tensor factor.  Similarly $M_{0,i}^Y(X)$ denotes the completion of $X$ whenever the action of $Y$ is on the $i$-th tensor factor. To illustrate this convention, consider the vector space $X=Y\otimes A\otimes Y$ as a $Y$-bimodule via the product on the third tensor factor, thus  $M_{0,3}^Y(Y\otimes A\otimes Y)$ is its completion. 
\end{rem}

%\begin{ex} [\cite{Action}] Sejam $A$ uma álgebra de Hopf de multiplicadores e $A'$ seu espaço dual. Definimos a aplicação
%\begin{eqnarray*}
%\triangleright : A\otimes A' &\longrightarrow & A'\\
%a\otimes\omega &\longmapsto & \omega(\_ a),
%\end{eqnarray*}
%então $A'$ é um $A$-módulo à esquerda não degenerado, mas não é unitário.
%\end{ex}

\subsection{Global Comodule Coalgebra}

\quad \
In this section, the notion of global comodule coalgebra is presented. For more details, we refer \cite{Action}, \cite{Galois} and \cite{Lydia}. For what follows, we shall assume that $Y$ is a nondegenerate algebra and $A$ is a regular multiplier Hopf algebra.

\begin{defi}\cite{Lydia}\label{comodulo}
We say that $Y$ is a \textit{left $A$-comodule}, if there is an injective linear map $\rho:Y\longmapsto M(A\otimes Y)$ such that:
	\begin{enumerate}
		\item[(i)]$\rho(Y)(A\otimes1)\subseteq A\otimes Y$\ \ and\ \ $(A\otimes1)\rho(Y)\subseteq A\otimes Y;$
		
		\item[(ii)]$(\imath_A\otimes\rho)\rho=(\Delta_A\otimes\imath_Y)\rho$\ \ \textit{(coassociativity)}.	
	\end{enumerate}	
\end{defi}

Similarly, we can define a \textit{right $A$-comodule}. In that case, the map $\rho$ is called a (global) coaction of $A$ on $Y$.  

\begin{rem}\label{extrho} The right side in (ii) makes sense because  $\Delta_A\otimes\imath_Y$ is a nondegenerate homomorphism, so we can extend it to $M(A\otimes Y)$. On the other hand, to justify the left side in (ii), we use (i), and define $(\imath_A\otimes\rho)(\rho(y))$ as a multiplier in $M_{0}(A\otimes A\otimes Y)\subseteq M(A\otimes A\otimes Y)$, for all $y\in Y$, as follows:
	\begin{eqnarray*}
		\overline{(\imath_A\otimes \rho)(\rho(y))}(a\otimes b)&=&(\imath_A\otimes \rho)(\rho(y))(a\otimes b\otimes 1)\\
		&=& (\imath_A\otimes \rho)(\rho(y)(a\otimes 1))(1\otimes b\otimes 1),\\
		\overline{\overline{(\imath_A\otimes \rho)(\rho(y))}}(a\otimes b)&=&	(a\otimes b\otimes 1)(\imath_A\otimes \rho)(\rho(y))\\
		&=& (1\otimes b\otimes 1)(\imath_A\otimes \rho)((a\otimes1)\rho(y)),
	\end{eqnarray*}     
	for all $a,b \in A.$ Thus, the coassociativity can be expressed by	\begin{eqnarray*}\label{igualdadecoassoc}
		(\imath_A\otimes\rho)(\rho(y)(a\otimes 1))(1\otimes b\otimes 1)=(\Delta_A\otimes\imath_Y)(\rho(y))(a\otimes b\otimes 1),\\
		\label{extcoa}(1\otimes b\otimes 1)(\imath_A\otimes\rho)((a\otimes 1)\rho(y))=(a\otimes b\otimes 1)(\Delta_A\otimes\imath_Y)(\rho(y)),
	\end{eqnarray*}
	for all $y\in Y$, $a, b\in A$.
\end{rem}

%O próximo resultado nos mostra que dada uma coação $\rho: Y\longrightarrow M(A\otimes Y)$, a injetividade de $\rho$ é equivalente a propriedade da counidade. 

\begin{pro}\textnormal{\cite{Galois}}\label{epsiloncoacao} If $Y$ is a left $A$-comodule via $\rho$, then $(\varepsilon_A\otimes \imath_Y)(\rho(y))=y$, for all $y\in Y$.
\end{pro}

Note that, $(\varepsilon_A\otimes \imath_Y)(\rho(y))$ is well defined since  $\varepsilon_A$ is a nondegenerate homomorphism. Therefore we can extend the map $\varepsilon_A\otimes \imath_Y$ to $M(A\otimes Y)$.

\begin{rem} We use Sigma notation (without sum) to express items (i) and (ii) of Definition \ref{comodulo}:
	\begin{enumerate}
		\item[(1)] $\rho(y)(a\otimes 1)= y^{(-1)}a\otimes y^{(0)}\in A\otimes Y$\quad and \quad $(a\otimes 1)\rho(y)= ay^{(-1)}\otimes y^{(0)}\in A\otimes Y$;
		\item[(2)] $ ay^{(-1)}\otimes by^{(0)(-1)}\otimes y^{(0)(0)}=\sum_i b_i(a_iy^{(-1)})_{(1)}\otimes (a_iy^{(-1)})_{(2)}\otimes y^{(0)}$, denoting $a\otimes b=\sum_i (b_i\otimes1)\Delta(a_i)$.
	\end{enumerate}
	\label{notsigma_global}
\end{rem}

%Algumas propriedades são obtidas usando a notação Sigma.

%\begin{pro}\textnormal{\cite{Lydia}} Para todos $y\in Y$ e $a\in A$, temos:
%	\begin{enumerate}
%		\item[(i)] $y^{00}\otimes y^{0-1}S_A^{-1}(y^{-1})a=y\otimes a$;
%		\item[(ii)] $y^{00}\otimes S_A^{-1}(y^{0-1})y^{-1}a=y\otimes a$.
%	\end{enumerate}
%\end{pro}
%\quad \ A seguir, apresentamos a definição de comódulo coálgebra no contexto de álgebras de Hopf de multiplicadores. Esta noção foi introduzida por L. Delvaux em \cite{Lydia}, no qual seu principal propósito era a construção de um coproduto smash entre duas álgebras de Hopf de multiplicadores. Antes de apresentar a definição e o teorema demonstrado por L. Delvaux, veremos alguns resultados preliminares e notações que serão utilizadas nessa seção e ao longo do texto. Continuaremos supondo que $Y$ é  uma álgebra com produto não degenerado e $A$ é uma álgebra de Hopf de multiplicadores regular.

Let $Y$ be a left $A$-comodule via $\rho$, define 
\begin{eqnarray}\label{defdaaplT}
	T: Y\otimes A &\longrightarrow & A\otimes Y\\
	y\otimes a &\longmapsto & \rho(y)(a\otimes 1)\nonumber.
\end{eqnarray}

The map $T$ is well defined since $\rho$ is bijective (see \cite{Lydia}).
%Além disso, aplicação linear $T$ é bijetiva com inversa dada por:
%$$T^{-1}(b\otimes z)=\sum_{z\in Y}=z^{0}\otimes S_A^{-1}(z^{-1})b,$$
%$b\in A$, $z\in Y$  (\cite{Galois}).
\begin{pro}\textnormal{\cite{Lydia}}\label{prop.1.5Lydia} Let $Y$ be a left $A$-comodule via $\rho$. Then,
	$(T\otimes\imath_A)(y\otimes\Delta_A(a))\in M_{0,1}^A(A\otimes Y\otimes A).$
	%	  é ''coberto" por elementos $(a'\otimes1\otimes 1)$ para todo $a'\in A$. Isto significa que  $(a'\otimes1\otimes 1)(T\otimes\imath_A)(y\otimes\Delta_A(a))\in A\otimes Y\otimes A$ e $(T\otimes\imath_A)(y\otimes\Delta_A(a))(a'\otimes1\otimes 1)\in A\otimes Y\otimes A.$			
\end{pro}
\begin{proof}
It is enough to define,
	\begin{eqnarray*}
		((T\otimes\imath_A)(y\otimes\Delta_A(a)))(a'\otimes 1\otimes1)&=&\sum y^{(-1)}a_{(1)}a'\otimes y^{(0)}\otimes a_{(2)},
	\end{eqnarray*} 
	where $a_{(1)}$ is covered by $a'$ and $y^{(-1)}$ is covered by $a_{(1)}a',$ and
	\begin{eqnarray*}
		(a'\otimes 1\otimes1)((T\otimes\imath_A)(y\otimes\Delta_A(a)))&=&\sum a'y^{(-1)}a_{(1)}\otimes y^{(0)}\otimes a_{(2)},
	\end{eqnarray*} 
	where $y^{(-1)}$ is covered by $a'$ and $a_{(1)}$ is covered by $a'y^{(-1)}.$
\end{proof}

\begin{pro}\label{paraequiv}\textnormal{\cite{Lydia}} Let $Y$ be a left $A$-comodule via $\rho$. Then,
	\begin{enumerate}
		\item [(i)] $(\imath_A\otimes T)((a'\otimes 1\otimes 1)(T\otimes \imath_A)(y\otimes\Delta_A(a))(1\otimes a''\otimes 1)=(\imath_A\otimes\rho)((a'\otimes1)\rho(y))(\Delta_A(a)(1\otimes a'')\otimes 1),$
		for all $a,a',a''\in A$, $y\in Y.$
		\item[(ii)]	$(\imath_A\otimes\rho)\rho=(\Delta_A\otimes\imath_Y)\rho$ if and only if $(\imath_A\otimes T)(T\otimes \imath_A)(\imath_Y\otimes\Delta_A)=(\Delta_A\otimes\imath_Y)T.$
		\end{enumerate}
	\end{pro}

%\begin{pro}\label{equivalLydia}\textnormal{\cite{Lydia}} Sejam $Y$ um $A$-comódulo à esquerda via $\rho$ e $T$ como definida em (\ref{defdaaplT}). Então	
%
%\end{pro}

\begin{defi}\cite{Lydia}\label{comodcoalglob} Let $Y$ be a multiplier Hopf algebra. We say that $Y$ is a  \textit{left} $A$-\textit{comodule coalgebra} if $Y$ is a left $A$-comodule via $\rho$ and
	\begin{enumerate}
		\item[(i)]\label{suposicao} $((\imath_Y\otimes T)(\Delta_Y(y)\otimes a))(1\otimes 1\otimes y')\in Y\otimes A\otimes Y$ and $(1\otimes 1\otimes y')((\imath_Y\otimes T)(\Delta_Y(y)\otimes a))\in Y\otimes A\otimes Y$
		
		\item[(ii)]$(\imath_A\otimes\Delta_Y)T(y\otimes a)=(T\otimes\imath_Y)(\imath_Y\otimes T)(\Delta_Y(y)\otimes a),$
	\end{enumerate}
	for all $y,y'\in Y$, $a\in A.$
\end{defi}
%\begin{rem}\label{paracaracterizacao1}
%	Podemos definir, para todo $y\in Y$, $(\imath_A\otimes\varepsilon_Y)\rho(y)$ como um multiplicador em $M(A)$ da seguinte forma: 	
%	\begin{eqnarray*}
%		\overline{(\imath_A\otimes\varepsilon_Y)\rho(y)}(a)&=&((\imath_A\otimes\varepsilon_Y)\rho(y))a\\
%		&=& (\imath_A\otimes\varepsilon_Y)(\rho(y)(a\otimes 1)),\\
%		\overline{\overline{(\imath_A\otimes\varepsilon_Y)\rho(y)}}(a)&=&a((\imath_A\otimes\varepsilon_Y)\rho(y))\\
%		&=& (\imath_A\otimes\varepsilon_Y)((a\otimes 1)\rho(y)).
%	\end{eqnarray*}     
%\end{rem}

\begin{pro}\textnormal{\cite{Lydia}}\label{paracaracterizacao} Let $Y$ be a left $A$-comodule coalgebra via $\rho$. Then, $(\imath_A\otimes\varepsilon_Y)(\rho(y)(a\otimes 1))=\varepsilon_Y(y)a,$ for all $y\in Y$, $a\in A$.
\end{pro}

Below, the theorem demonstrated by L. Delvaux in \cite{Lydia} is presented regarding the construction of the smash coproduct.  

\begin{thm}\textnormal{\cite{Lydia}}
	Let $Y$ and $A$  be regular multiplier Hopf algebras such that $A$ is commutative. If $Y$ is a left $A$-comodule coalgebra via a homomorphism $\rho$, then $\overline{\Delta}$ is a  homomorphism on $Y\otimes A$ such that $(Y\otimes A,\overline{\Delta})$ is a regular multiplier Hopf algebra.
	\begin{proof}
It is sufficient to define the comultiplication $\overline{\Delta}$ on $Y\otimes A$ by  
\begin{eqnarray*}
	\overline{\Delta}(y\otimes a)((y'\otimes a')\otimes (y''\otimes a''))&=& (((\imath_Y\otimes T)(\Delta_Y(y)\otimes a_{(1)}))(1^2\otimes y'')\otimes a_{(2)}a'')((y'\otimes a')\otimes(1^2)),\\
	((y'\otimes a')\otimes (y''\otimes a''))\overline{\Delta}(y\otimes a)&=& ((1^2)\otimes(y''\otimes a''))(y'y_{(1)}\otimes(a'\otimes1^2)((T\otimes\imath_A)(y_{(2)}\otimes\Delta_A(a)))),
\end{eqnarray*}
for all $y\in Y$, $a\in A$, $y',y''\in Y$ and $a',a''\in A$.
	\end{proof}
	
\end{thm}

\section{Partial Comodule Coalgebra}
%%%%%%%%%%%%%%%%%%%%%%%%%%%%%%%%%%%%%%%

\subsection{Partial Comodule Coalgebra}

In this section, we present the definition of a left partial  comodule coalgebra in the context of multiplier Hopf algebras. We show that this definition generalizes the Hopf case for partial comodule coalgebras presented in \cite {Batista}, and the global case of the multiplier Hopf algebras in \cite {Lydia}. Throughout this section, we assume that $Y$ and $A$ are regular multiplier Hopf algebras and the partial $A$-comodule coalgebras will be consider on the left. The partial comodule coalgebras on the right are defined similarly.

\begin{defi} \label{part_comod_coal}\cite{Batista} Let $A$ be a Hopf algebra and let $Y$ be a coalgebra. We say that $Y$ is a \textit{partial $A$-comodule coalgebra} if there exists a linear map $\rho: Y \longrightarrow  A\otimes Y$ such that
	\begin{enumerate}
		\item [(i)] $(\varepsilon_A\otimes \imath_Y)\rho=\imath_Y$
		\item [(ii)] $(\imath_A\otimes\Delta_Y)\rho=(m_A\otimes\imath_Y\otimes\imath_Y)(\imath_A\otimes \tau_{Y,A}\otimes \imath_Y)(\rho\otimes\rho)\Delta_Y$
		\item[(iii)] $(\imath_A\otimes\rho)\rho=(m_A\otimes\imath_A\otimes\imath_Y)\{(\imath_A\otimes\varepsilon_Y)\rho\otimes[(\Delta_A\otimes\imath_Y)\rho]\}\Delta_Y$,
	\end{enumerate}
where $\tau_{Y,A}$ is the twist map. Furthermore, we call $Y$ a \textit{symmetric partial $A$-comodule coalgebra} if the following additional condition also holds:
\begin{enumerate}
	\item[(iv)]$(\imath_A\otimes\rho)\rho=(m_A\otimes\imath_A\otimes\imath_Y)(\imath_A\otimes\tau_{A\otimes Y,A})\{[(\Delta_A\otimes\imath_Y)\rho]\otimes(\imath_A\otimes\varepsilon_Y)\rho\}\Delta_Y$.
	\end{enumerate}
\end{defi}

 This notion can be extended to the context of regular multiplier Hopf algebra as follows.
\begin{defi}\label{defcomodcopar} We say that $Y$ is a \textit{partial $A$-comodule coalgebra} if there exists a linear map
	\begin{eqnarray*}
		\rho: Y & \longrightarrow & M(A\otimes Y)
	\end{eqnarray*}
	such that, for all $y, y'\in Y$, $a\in A$,
	\begin{itemize}
		\item[(i)] $(\varepsilon_A\otimes \imath_Y)\rho(y)=y$
		\item[(ii)] $((\imath_Y\otimes T)(\Delta_Y(y)\otimes a))(1\otimes 1\otimes y')\in Y\otimes A\otimes Y$ and $(1\otimes 1\otimes y')((\imath_Y\otimes T)(\Delta_Y(y)\otimes a))\in Y\otimes A\otimes Y$
		
		\item[(iii)] $(\imath_A\otimes\Delta_Y)T=(T\otimes \imath_Y)(\imath_Y\otimes T)(\Delta_Y\otimes\imath_A)$
		
		\item[(iv)]$(\imath_A\otimes T)(T\otimes\imath_A)(\imath_Y\otimes\Delta_A)=((\imath_A\otimes\varepsilon_Y)T\otimes\imath_A\otimes\imath_Y)(\imath_Y\otimes(\Delta_A\otimes\imath_Y)T)(\Delta_Y\otimes\imath_A)$,
	\end{itemize}
	where $	T :Y\otimes A  \longrightarrow A\otimes Y$	is given by $T(y\otimes a)=\rho(y)(a\otimes 1)\in A\otimes Y.$
	
	Moreover, $Y$ is a \textit{symmetric partial $A$-comodule coalgebra} if the following additional conditions hold
	\begin{enumerate}    	
		\item[(v)]$(\overline{T}\otimes\imath_Y)(a\otimes\Delta_Y(y))(1\otimes y'\otimes 1)\in A\otimes Y\otimes Y$
		
		\item[(vi)]$(\imath_A\otimes\overline{T})(\imath_A\otimes\tau_{Y\otimes A})(\overline{T}\otimes\imath_A)(\imath_A\otimes\tau_{A\otimes Y})(\Delta_A\otimes\imath_Y)=((\imath_A\otimes\varepsilon_Y)\overline{T}\otimes\imath_A\otimes\imath_Y)(\imath_A\otimes\tau_{A\otimes Y,Y})((\Delta_A\otimes\imath_Y)\overline{T}\otimes\imath_Y)(\imath_A\otimes\Delta_Y),$
	\end{enumerate}
	where $\overline{T} :A\otimes Y  \longrightarrow A\otimes Y$
	is given by $\overline{T}(a\otimes y)=(a\otimes 1)\rho(y)\in A\otimes Y$.
\end{defi} 	
 	\begin{rem}\label{trocaLydiaparcial} Let $Y$ be a partial $A$-comodule coalgebra via $\rho$. Note that 
 		$(T\otimes\imath_A)(y\otimes\Delta_A(a))\in M_{0,1}^A(A\otimes Y\otimes A),$ considering
 		\begin{eqnarray*}
 			((T\otimes\imath_A)(y\otimes\Delta_A(a)))(a'\otimes 1\otimes1)&=&\sum y^{(-1)}a_{(1)}a'\otimes y^{(0)}\otimes a_{(2)},
 		\end{eqnarray*} 
 		where $a_{(1)}$ is covered by $a'$ and $y^{(-1)}$ is covered by $a_{(1)}a',$ and
 		\begin{eqnarray*}
 			(a'\otimes 1\otimes1)((T\otimes\imath_A)(y\otimes\Delta_A(a)))&=&\sum a'y^{(-1)}a_{(1)}\otimes y^{(0)}\otimes a_{(2)},
 		\end{eqnarray*} 
 		where $y^{(-1)}$ is covered by $a'$ and $a_{(1)}$ is covered by $a'y^{(-1)}.$
 	\end{rem}
\begin{rem}
	\begin{enumerate} 
		\item [(1)] In Definition \ref{defcomodcopar} item (ii) is necessary to give sense to (iii). First, we can rewrite item (iii), for every, $y\in Y$, $a\in A$, $y'\in Y$, as follows:
		\begin{eqnarray}\label{justifitemii}
		y^{(-1)}a\otimes\Delta_Y(y^{(0)})(1\otimes y')=(T\otimes\imath_Y)((\imath_Y\otimes T)(\Delta_Y(y)\otimes a))(1\otimes 1\otimes y').
	\end{eqnarray}
	Thus, using item (ii), denoting $z=(\imath_Y\otimes T)(\Delta_Y(y)\otimes a),$ we can define the multiplier $(T\otimes \imath_Y)z\in M_{0,3}^Y(A\otimes Y\otimes Y)$ as
	\begin{eqnarray*}
		\overline{(T\otimes\imath_Y)(z)}(y')=(T\otimes\imath_Y)(z)(1\otimes 1\otimes y')
		= (T\otimes\imath_Y)(z(1\otimes 1\otimes y'))\\
		\overline{\overline{(T\otimes\imath_Y)(z)}}(y')=(1\otimes 1\otimes y')(T\otimes\imath_Y)(z)
		= (T\otimes\imath_Y)((1\otimes 1\otimes y')z).
	\end{eqnarray*}     		
	
%	Vamos verificar a compatibilidade em $M_{0,3}^Y(A\otimes Y\otimes Y)$.
%	\begin{eqnarray*}
%		y'\overline{(T\otimes\imath_Y)(z)}(y'')&=&(1\otimes 1\otimes y')((T\otimes\imath_Y)(z)(1\otimes 1\otimes y''))\\
%		&=& (1\otimes 1\otimes y')((T\otimes\imath_Y)(z(1\otimes 1\otimes y'')))\\
%		&\stackrel{\ref{defcomodcopar}(ii)}{=}& (1\otimes 1\otimes y')((T\otimes\imath_Y)(x\otimes b\otimes t))\\
%		&=& (1\otimes 1\otimes y')(T(x\otimes b)\otimes t)\\
%		&=& T(x\otimes b)\otimes y't\\
%		&=& (T\otimes\imath_Y)(x\otimes b\otimes y't)\\
%		&=& (T\otimes\imath_Y)((1\otimes 1\otimes y')(x\otimes b\otimes t))\\
%		&=& (T\otimes\imath_Y)((1\otimes 1\otimes y')(z(1\otimes 1\otimes y''))\\
%		&=& (T\otimes\imath_Y)(((1\otimes 1\otimes y')z)(1\otimes 1\otimes y''))\\
%		&=& (T\otimes\imath_Y)((p\otimes c\otimes w)(1\otimes 1\otimes y''))\\
%		&=& T(p\otimes c)\otimes w)(1\otimes 1\otimes y'')\\
%		&=& T(p\otimes c)(1\otimes 1\otimes wy'')\\
%		&=& ((T\otimes\imath_Y)(p\otimes c\otimes w))(1\otimes 1\otimes y'')\\
%		&=& ((T\otimes\imath_Y)((1\otimes 1\otimes y')z)(1\otimes 1\otimes y'')\\
%		&=& \overline{\overline{(T\otimes\imath_Y)(z)}}(y')y'',	
%	\end{eqnarray*}
%	para todos $y',y''\in Y.$   
	
	Therefore, the right side of  (\ref{justifitemii}) can be seen as
	\begin{eqnarray}\label{comppar2}
		(T\otimes\imath_Y)((\imath_Y\otimes T)(\Delta_Y(y)\otimes a))(1\otimes 1\otimes y')=(T\otimes\imath_Y)((\imath_Y\otimes T)(\Delta_Y(y)\otimes a)(1\otimes 1\otimes y')).
	\end{eqnarray}
	\ \
	\ \
\item[(2)]
 $(T\otimes\imath_A )(y\otimes\Delta_A(a))$  is a multiplier in $M_{0,3}^A(A\otimes Y\otimes A)$ for all $y\in Y$, $a\in A$, defined by 
	\begin{eqnarray*}
		\overline{(T\otimes\imath_A)(y\otimes\Delta_A(a))}(1\otimes 1\otimes a')&=&(T\otimes\imath_A)(y\otimes\Delta_A(a))(1\otimes 1\otimes a')\\
		&=& (T\otimes\imath_A)(y\otimes\Delta_A(a)(1\otimes a')),\\     	
		\overline{\overline{(T\otimes\imath_A)(y\otimes\Delta_A(a))}}(1\otimes 1\otimes a')&=&(1\otimes 1\otimes a')(T\otimes\imath_A)(y\otimes\Delta_A(a))\\
		&=&(T\otimes\imath_A)(y\otimes(1\otimes a')\Delta_A(a)).
	\end{eqnarray*}
	
	Similarly, we define the multiplier $(\imath_A\otimes T)((T\otimes\imath_A)(y\otimes\Delta_A(a)))\in M_{0,1}^A(A\otimes A\otimes Y).$
\item[(3)] For every $a'\in A$, $y\in Y$, 
	\begin{eqnarray}\label{defparobs3}\label{compondopar}
		[(\imath_A\otimes T)(T\otimes\imath_A)(y\otimes\Delta_A(a))](1\otimes a'\otimes 1)=(\imath_A\otimes T)((T\otimes\imath_A)(y\otimes\Delta_A(a))(1\otimes 1\otimes a')).
	\end{eqnarray}
	Indeed, for all $a''\in A$,
	\begin{eqnarray*}
		[(\imath_A\otimes T)(T\otimes\imath_A)(y\otimes\Delta_A(a))](a''\otimes a'\otimes 1)&=&(\imath_A\otimes T)[(T\otimes\imath_A)(y\otimes\Delta_A(a))(a''\otimes 1\otimes 1)](1\otimes a'\otimes 1)\\
		&\stackrel{\ref{trocaLydiaparcial}}{=}&((\imath_A\otimes T)(y^{(-1)}a_{(1)}a''\otimes y^{(0)}\otimes a_{(2)}))(1\otimes a'\otimes 1)\\
		&=&y^{(-1)}a_{(1)}a''\otimes y^{(0)(-1)}a_{(2)}a'\otimes y^{(0)(0)}\\
		&=&(\imath_A\otimes T)((T\otimes\imath_A)(y\otimes\Delta_A(a)(1\otimes a')(a''\otimes 1)))\\
		&=&(\imath_A\otimes T)((T\otimes\imath_A)(y\otimes\Delta_A(a)(1\otimes a')))(a''\otimes 1\otimes 1)\\
		&=&[(\imath_A\otimes T)((T\otimes\imath_A)(y\otimes\Delta_A(a))(1\otimes1\otimes a'))](a''\otimes 1\otimes 1). 
	\end{eqnarray*}
	
%	Portanto, provamos que: 	
%	\begin{eqnarray}
%		&&\hspace{-3cm}[(\imath_A\otimes T)(T\otimes\imath_A)(y\otimes\Delta(a))](1\otimes a'\otimes 1)=[(\imath_A\otimes T)((T\otimes\imath_A)(y\otimes\Delta(a))(1\otimes1\otimes a'))], 
%	\end{eqnarray} 
%	para todos $a'\in A$, $y\in Y$.
%	   	E, disto, segue a seguinte igualdade de multiplicadores:
%	     	\begin{eqnarray}
%	     	&&\hspace{-2cm}[(\imath_A\otimes T)(T\otimes\imath_A)(y\otimes\Delta(a))](1\otimes a'\otimes 1)=\nonumber\\
%	     	\hspace{1cm}&=&[(\imath_A\otimes T)((T\otimes\imath_A)(y\otimes\Delta(a))(1\otimes1\otimes a'))]. 
%	     	\end{eqnarray} 
	     	
	     	\ \
	     	\ \
	
\item[(4)] Item (iv) by Definition \ref{defcomodcopar} can be written explicitly covering both sides by $(1\otimes a'\otimes y')$, $a'\in A$, $y'\in Y$, as follows:
	\begin{eqnarray}\label{justifitemiv}\label{reescreve}
		y^{(-1)}a_{(1)}\otimes y^{(0)(-1)}a_{(2)}a'\otimes y^{(0)(0)}y'
		= x^{(-1)}b_{(1)}\varepsilon_Y(x^{(0)})\otimes b_{(2)}a'\otimes w,
	\end{eqnarray}
	where $(\imath_Y\otimes T)(\Delta_Y(y)\otimes a)(1^2\otimes y')=x\otimes b\otimes w$, by item (ii).
	In fact,
	\begin{eqnarray*}
		y^{(-1)}a_{(1)}\otimes y^{(0)(-1)}a_{(2)}a'\otimes y^{(0)(0)}y'&=&(\imath_A\otimes T)[(T\otimes\imath_A)(y\otimes\Delta_A(a)(1\otimes a'))](1\otimes 1\otimes y')\\
		&=&(\imath_A\otimes T)[(T\otimes\imath_A)(y\otimes\Delta_A(a))(1\otimes 1\otimes a')](1\otimes 1\otimes y')\\
		&\stackrel{(\ref{defparobs3})}{=}&[(\imath_A\otimes T)(T\otimes\imath_A)(y\otimes\Delta_A(a))](1\otimes a'\otimes y')\\
		&\stackrel{\ref{defcomodcopar}(iv)}{=}&((\imath_A\otimes\varepsilon_Y)T\otimes\imath_A\otimes\imath_Y)(\imath_Y\otimes(\Delta_A\otimes\imath_Y)T)(\Delta_Y(y)\otimes a)(1\otimes a'\otimes y')\\
		&=&((\imath_A\otimes\varepsilon_Y)T\otimes\imath_A\otimes\imath_Y)\\
		&&((\imath_Y\otimes\Delta_A\otimes\imath_Y)[(\imath_Y\otimes T)(\Delta_Y(y)\otimes a)(1^2\otimes y')](1^2\otimes a'\otimes1))\\
		&\stackrel{\ref{defcomodcopar} (ii)}{=}&((\imath_A\otimes\varepsilon_Y)T\otimes\imath_A\otimes\imath_Y)((x\otimes\Delta_A(b)\otimes w)(1\otimes1\otimes a'\otimes1))\\
		&=&((\imath_A\otimes\varepsilon_Y)T\otimes\imath_A\otimes\imath_Y)(x\otimes\Delta_A(b)(1\otimes a')\otimes w)\\
		&=&x^{(-1)}b_{(1)}\varepsilon_Y(x^{(0)})\otimes b_{(2)}a'\otimes w. 			
	\end{eqnarray*}

\end{enumerate}
\end{rem}

The next Lemma gives us an equivalence similar to the result presented in Proposition \ref{paraequiv} item (ii).

\begin{lem} Let $Y$ be a symmetric partial $A$-comodule coalgebra via $\rho$. Then the statements below are equivalent:
	\begin{enumerate}
		\item[(i)]$(\imath_A\otimes\rho)\rho=(\Delta_A\otimes\imath_Y)\rho;$
		\item[(ii)]$(\imath_A\otimes\overline{T})(\imath_A\otimes\tau_{Y\otimes A})(\overline{T}\otimes\imath_A)(\imath_A\otimes\tau_{A\otimes Y})(\Delta_A\otimes\imath_Y)=(\Delta_A\otimes\imath_Y)\overline{T}.$
	\end{enumerate}
\end{lem}
\begin{proof} Suppose that $(\imath_A\otimes\rho)\rho=(\Delta_A\otimes\imath_Y)\rho.$ Then, given $a,b,b'\in A$, $y\in Y$, 
	\begin{eqnarray*}	
		& \ &(b\otimes1\otimes1)(\Delta_A\otimes\imath_Y)\overline{T}(a\otimes y)(b'\otimes1\otimes1)\\
		&=&(b\otimes1\otimes1)(\Delta_A\otimes\imath_Y)((a\otimes1)\rho(y))(b'\otimes1\otimes1)\\
		&=&((b\otimes1)\Delta_A(a)\otimes1)(\Delta_A\otimes\imath_Y)\rho(y)(b'\otimes1\otimes1)\\		
		&=&(ba_{(1)}\otimes a_{(2)}\otimes 1)(\imath_A\otimes\rho)\rho(y)(b'\otimes1\otimes1)\\
		&=&ba_{(1)}y^{(-1)}b'\otimes(a_{(2)}\otimes1)\rho(y^{(0)})\\
		&=&ba_{(1)}y^{(-1)}b'\otimes\overline{T}(a_{(2)}\otimes y^{(2)})\\
		&=&(\imath_A\otimes\overline{T})(\imath_A\otimes\tau_{Y\otimes A})((ba_{(1)}\otimes1)(\rho(y)(b'\otimes1))\otimes a_{(2)})\\
		&=&(\imath_A\otimes\overline{T})(\imath_A\otimes\tau_{Y\otimes A})((\overline{T}(ba_{(1)}\otimes y)\otimes a_{(2)})(b'\otimes1\otimes1))\\
		&=&((\imath_A\otimes\overline{T})(\imath_A\otimes\tau_{Y\otimes A})(\overline{T}\otimes\imath_A)(ba_{(1)}\otimes y\otimes a_{(2)}))(b'\otimes1\otimes1)\\
		&=&((\imath_A\otimes\overline{T})(\imath_A\otimes\tau_{Y\otimes A})(\overline{T}\otimes\imath_A)(\imath_A\otimes\tau_{A\otimes Y})(ba_{(1)}\otimes a_{(2)}\otimes y))(b'\otimes1\otimes1)\\
		&=&((\imath_A\otimes\overline{T})(\imath_A\otimes\tau_{Y\otimes A})(\overline{T}\otimes\imath_A)(\imath_A\otimes\tau_{A\otimes Y})((b\otimes1)\Delta_A(a)\otimes y))(b'\otimes1\otimes1)\\
		&=&(b\otimes1\otimes1)((\imath_A\otimes\overline{T})(\imath_A\otimes\tau_{Y\otimes A})(\overline{T}\otimes\imath_A)(\imath_A\otimes\tau_{A\otimes Y})(\Delta_A\otimes\imath_Y)(a\otimes y))(b'\otimes1\otimes1).
	\end{eqnarray*}
	
Conversely,
	\begin{eqnarray*}
		& \ &	((b\otimes1)\Delta_A(a)\otimes1)(\imath_A\otimes\rho)\rho(y)(b'\otimes1\otimes1)\\
		&=&((b\otimes1)\Delta_A(a)\otimes1)(\imath_A\otimes\rho)(\rho(y)(b'\otimes1))\\
		&=&ba_{(1)}y^{(-1)}b'\otimes \overline{T}(a_{(2)}\otimes y^{(0)})\\
		&=&(\imath_A\otimes\overline{T})(\imath_A\otimes\tau_{Y\otimes A})((ba_{(1)}\otimes1)(\rho(y)(b'\otimes1))\otimes a_{(2)})\\
		&=&((\imath_A\otimes\overline{T})(\imath_A\otimes\tau_{Y\otimes A})(\overline{T}\otimes\imath_A)(\imath_A\otimes\tau_{A\otimes Y})((b\otimes1)\Delta_A(a)\otimes y))(b'\otimes1\otimes1)\\
		&=&(b\otimes1\otimes1)((\imath_A\otimes\overline{T})(\imath_A\otimes\tau_{Y\otimes A})(\overline{T}\otimes\imath_A)(\imath_A\otimes\tau_{A\otimes Y})(\Delta_A\otimes\imath_Y)(a\otimes y))(b'\otimes1\otimes1)\\
		&=&(b\otimes1\otimes1)(\Delta_A\otimes\imath_Y)\overline{T}(a\otimes y)(b'\otimes1\otimes1)\\
		&=&(b\otimes1\otimes1)(\Delta_A\otimes\imath_Y)((a\otimes1)\rho(y))(b'\otimes1\otimes1)\\
		&=&((b\otimes1)\Delta_A(a)\otimes1)(\Delta_A\otimes\imath_Y)\rho(y)(b'\otimes1\otimes1),
	\end{eqnarray*}
	for all $a,b,b'\in A$, $y\in Y.$ Since $(A\otimes1)\Delta_A(A)$ generates $A\otimes A$, item (i) holds.  
%\begin{eqnarray*}
%	  			((b\otimes1)\Delta_A(a)\otimes1)(\Delta_A\otimes\imath_Y)\rho(y)(b'\otimes1\otimes1)\Leftrightarrow\\
%		  			&&\hspace{-2cm}\Leftrightarrow((b\otimes1)\Delta_A(a)\otimes1)(\imath_A\otimes\rho)(\rho(y)(b'\otimes1))=((b\otimes1\otimes1)(\Delta_A\otimes\imath_Y)(a\otimes1))(\Delta_A\otimes\imath_Y)\rho(y)(b'\otimes1\otimes1)\\
%		  			&&\hspace{-2cm}\Leftrightarrow((b\otimes1)\Delta_A(a)\otimes1)(\imath_A\otimes\rho)(y^{-1}b'\otimes y^{0})=(b\otimes1\otimes1)(\Delta_A\otimes\imath_Y)((a\otimes1)\rho(y))(b'\otimes1\otimes1)\\
%		  			&&\hspace{-2cm}\Leftrightarrow(ba_1\otimes a_2\otimes1)(y^{-1}b'\otimes\rho( y^{0}))=(b\otimes1\otimes1)(\Delta_A\otimes\imath_Y)\overline{T}(a\otimes y)(b'\otimes1\otimes1),
%	\end{eqnarray*}

\end{proof}

\begin{pro}Every comodule coalgebra is a symmetric partial comodule coalgebra. 
\end{pro}
\begin{proof}
	Assume that $Y$ is an $A$-comodule coalgebra via $\rho$. By Proposition \ref{epsiloncoacao} we have that $(\varepsilon_A\otimes\imath_Y)\rho=\imath_Y,$ \textit i.e., item (i) of Definition \ref{defcomodcopar} holds. The items (ii) and (iii) follow immediately from Definition \ref{comodcoalglob}. Thus it is enough to show item (iv). Indeed,	 		
	\begin{eqnarray*}
		& \ &[((\imath_A\otimes\varepsilon_Y)T\otimes\imath_A\otimes\imath_Y)(\imath_Y\otimes(\Delta_A\otimes\imath_Y)T)(\Delta_Y(y)\otimes a)](1\otimes a'\otimes y')\\
		&\stackrel{(\ref{reescreve})}{=}&x^{(-1)}b_{(1)}\otimes\varepsilon_Y(x^{(0)})\otimes b_{(2)}a'\otimes w\\
		&\stackrel{\ref{paracaracterizacao}}{=}&b_{(1)}\varepsilon_Y(x)\otimes b_{(2)}a'\otimes w\\
		&=&(\varepsilon_Y\otimes\imath_A\otimes\imath_A\otimes\imath_Y)(x\otimes \Delta_A(b)(1\otimes a')\otimes w)\\
		&=&(\varepsilon_Y\otimes\Delta_A\otimes\imath_Y)(x\otimes b\otimes w)(1\otimes a'\otimes 1)\\
		&=&(\varepsilon_Y\otimes\Delta_A\otimes\imath_Y)((\imath_Y\otimes T)(\Delta_Y(y)\otimes a)(1\otimes1\otimes y'))(1\otimes a'\otimes 1)\\
		&=&(\varepsilon_Y\otimes\Delta_A\otimes\imath_Y)((\imath_Y\otimes T)(\Delta_Y(y)\otimes a))(1\otimes a'\otimes y'))\\
		&=&(\Delta_A\otimes\imath_Y)((\varepsilon_Y\otimes T)(\Delta_Y(y)\otimes a))(1\otimes a'\otimes y')\\
		&=&(\Delta_A\otimes\imath_Y)(T((\varepsilon_Y\otimes\imath_Y)\Delta_Y(y)\otimes a))(1\otimes a'\otimes y')\\
		&=&(\Delta_A\otimes\imath_Y)(T(y\otimes a))(1\otimes a'\otimes y')\\
		&\stackrel{\ref{paraequiv}}{=}&(\imath_A\otimes T)(T\otimes\imath_A)(y\otimes\Delta_A(a))(1\otimes a'\otimes y'),	 		
	\end{eqnarray*}
	for all $a,a'\in A$, $y'\in Y$. The symmetric condition holds similarly. 	 	 	
\end{proof}

\begin{pro}Let $Y$ be a symmetric partial $A$-comodule coalgebra via $\rho$. Then, $Y$ is a $A$-comodule coalgebra if and only if 
	\begin{eqnarray}\label{eqcaract}(\imath_A\otimes\varepsilon_Y)\rho(y)=\varepsilon_Y(y)1_{M(A)},
	\end{eqnarray}
for all $y\in Y.$
\end{pro}
\begin{proof}
	If $Y$ is an $A$-comodule coalgebra, then (\ref{eqcaract}) follows by Proposition \ref{paracaracterizacao}.
	On the other side, if (\ref{eqcaract}) holds, then:
	
	$(i) $ $\rho$ é injective.
	
	Indeed, consider $y\in Y$ such that $\rho(y)=0,$ thus $0=(\varepsilon_A\otimes \imath_Y)\rho(y)\stackrel{\ref{defcomodcopar}(i)}{=}y.$
	
	$(ii)$\ \ $(\imath_A\otimes\rho)\rho=(\Delta_A\otimes\imath_Y)\rho.$
	
Using item $(ii)$ of Proposition \ref{paraequiv}, it is enough to show that $(\Delta_A\otimes\imath_Y)T=(\imath_A\otimes T)(T\otimes\imath_A)(\imath_Y\otimes\Delta_A).$
	\begin{eqnarray*}
		& \ &(\Delta_A\otimes\imath_Y)(T(y\otimes a))(1\otimes a'\otimes y')\\
		&=&(\Delta_A\otimes\imath_Y)(T((\varepsilon_Y\otimes\imath_Y)\Delta_Y(y)\otimes a))(1\otimes a'\otimes y')\\
		&=&(\Delta_A\otimes\imath_Y)((\varepsilon_Y\otimes T)(\Delta_Y(y)\otimes a))(1\otimes a'\otimes y')\\
		&=&(\varepsilon_Y\otimes\Delta_A\otimes\imath_Y)((\imath_Y\otimes T)(\Delta_Y(y)\otimes a)(1\otimes 1\otimes y'))(1\otimes a'\otimes 1)\\
		&=&(\varepsilon_Y\otimes\Delta_A\otimes\imath_Y)(x\otimes b\otimes w)(1\otimes a'\otimes 1)\\
		&=&\varepsilon_Y(x)\otimes\Delta_A(b)(1\otimes a')\otimes w\\
		&=&\varepsilon_Y(x)b_{(1)}\otimes b_{(2)} a'\otimes w\\
		&\stackrel{(\ref{eqcaract})}{=}&(\imath_A\otimes\varepsilon_Y)(\rho(x)(b_{(1)}\otimes1))\otimes b_{(2)} a'\otimes w\\
		&=& x^{(-1)}b_{(1)}\varepsilon_Y(x^{(0)})\otimes b_{(2)} a'\otimes w\\
		&=&((\imath_A\otimes\varepsilon_Y)T\otimes\imath_A\otimes\imath_Y)(\imath_Y\otimes(\Delta_A\otimes\imath_Y)T)(\Delta_Y(y)\otimes a)(1\otimes a'\otimes y')\\
		&\stackrel{\ref{defcomodcopar}(iv)}{=}&(\imath_A\otimes T)(T\otimes\imath_A)(y\otimes\Delta_A(a))(1\otimes a'\otimes y'),
	\end{eqnarray*}
	for all $y,y'\in Y$ and $a,a'\in A$, denoting $(\imath_Y\otimes T)(\Delta_Y(y)\otimes a)(1\otimes 1\otimes y')=x\otimes b\otimes w\in Y\otimes A\otimes Y.$ 	
	
	The other items of Definition \ref{comodcoalglob} follow immediately from Definition \ref{defcomodcopar}.
\end{proof}

\begin{pro}
Let $Y$ and $A$ be Hopf algebras. Then, Definition
	\ref{part_comod_coal} and Definition \ref{defcomodcopar} are equivalent.
\end{pro}
\begin{proof}
	Suppose that Definition \ref{part_comod_coal} holds. Then items (i) and (ii) of Definition \ref{defcomodcopar} are automatically satisfied. Furthermore,
	for every $y\in Y$, $a\in A$, 
	\begin{eqnarray*}
		(\imath_A\otimes\Delta_Y)T(y\otimes a)&=&(\imath_A\otimes\Delta_Y)(\rho(y)(a\otimes 1))\\
		&=& (\imath_A\otimes\Delta_Y)\rho(y)(a\otimes 1\otimes 1)\\
		&\stackrel{\ref{part_comod_coal}(ii)}{=}& (m_A\otimes\imath_Y\otimes\imath_Y)(\imath_A\otimes \tau_{Y,A}\otimes \imath_Y)(\rho\otimes\rho)\Delta_Y(y)(a\otimes 1\otimes 1)\\
		&=&(T\otimes \imath_Y)(y_{(1)}\otimes \rho(y_{(2)})(a\otimes 1))\\
		&=&(T\otimes \imath_Y)(\imath_Y\otimes T)(\Delta_Y(y)\otimes a),
	\end{eqnarray*}
		and,
	\begin{eqnarray*}
		(\imath_A\otimes T)(T\otimes\imath_A)(y\otimes\Delta_A(a))&=&(\imath_A\otimes T)(\rho(y)(a_{(1)}\otimes 1)\otimes a_{(2)})\\
		&=&((\imath_A\otimes\rho)\rho(y))(a_{(1)}\otimes a_{(2)}\otimes 1)\\
		&\stackrel{\ref{part_comod_coal} (iii)}{=}&(m_A\otimes\imath_A\otimes\imath_Y)\{(\imath_A\otimes\varepsilon_Y)\rho\otimes[(\Delta_A\otimes\imath_Y)\rho]\}\Delta_Y(y)(a_{(1)}\otimes a_{(2)}\otimes 1)\\
		&=&{y_{(1)}}^{(-1)}{{y_{(2)}}^{(-1)}}_{(1)}a_{(1)}\otimes {{y_{(2)}}^{(-1)}}_{(2)}a_{(2)}\otimes {y_{(2)}}^{(0)}\varepsilon_Y({y_{(1)}}^{(0)})\\
		&=&((\imath_A\otimes\varepsilon_Y)T\otimes\imath_A\otimes\imath_Y)(\imath_Y\otimes(\Delta_A\otimes\imath_Y)T)(\Delta_Y(y)\otimes a).	\end{eqnarray*} 
	
Conversely, just take $ a = 1_A $ in items (iii) and (iv) of Definition \ref{defcomodcopar}. 
	\end{proof}

\begin{pro}\label{comodcoalparcviah}
Consider the linear map
	\begin{eqnarray*}
		\rho: Y&\longrightarrow & M(A\otimes Y)\\
		y &\longmapsto & h\otimes y
	\end{eqnarray*}	
	where $h\in M(A)$. Then $Y$ is a partial $A$-comodule coalgebra via $\rho$ if and only if
	\begin{enumerate} 	
		\item[(i)]$\varepsilon(h)=1_\Bbbk$
		\item[(ii)]$h\otimes h=(h\otimes 1)\Delta(h).$
	\end{enumerate}
\end{pro}
Moreover, $Y$ is a symmetric partial $A$-comodule coalgebra via $\rho$ if and only if the following additional condition also holds 
\begin{enumerate} 	
	\item[(iii)]$h\otimes h=\Delta(h)(h\otimes 1).$
	\end{enumerate}
	\begin{proof}
	Suppose that $Y$ is a partial $A$-comodule coalgebra via $\rho$, for all $y\in Y$,
	$y\stackrel{\ref{defcomodcopar}(i)}{=}(\varepsilon\otimes \imath)\rho(y)=\varepsilon(h)y.$ Then, $\varepsilon(h)=1_\Bbbk.$
	
	To show item (ii), we will use item $ (iv) $ of Definition \ref{defcomodcopar}. For every $y,y'\in Y$, $a, a'\in A$, on one hand, 
	\begin{eqnarray*}
	(\imath_A\otimes T)(T\otimes\imath_A)(y\otimes\Delta_A(a))](1\otimes a'\otimes y')&\stackrel{(\ref{compondopar})}{=}& (\imath_A\otimes T)[(T\otimes\imath_A)(y\otimes\Delta_A(a)(1\otimes a'))](1\otimes1\otimes y')\\&=& (\imath_A\otimes T)(T(y\otimes a_{(1)}a')\otimes a_{(2)})(1\otimes1\otimes y')\\
	&=&(ha_{(1)}a'\otimes T(y\otimes a_{(2)}))(1\otimes1\otimes y')\\
	&=& ha_{(1)}a'\otimes ha_{(2)}\otimes yy'\\
	&=& (h\otimes h\otimes 1)(\Delta_A(a)(1\otimes a')\otimes yy').
	\end{eqnarray*}
	
	On the other hand, 
	\begin{eqnarray*}
	&\ &[((\imath_A\otimes\varepsilon_Y)T\otimes\imath_A\otimes\imath_Y)(\imath_Y\otimes(\Delta_A\otimes\imath_Y)T)(\Delta_Y(y)\otimes a)](1\otimes a'\otimes y')\\
	&=&((\imath_A\otimes\varepsilon_Y)T\otimes\imath_A\otimes\imath_Y)[(\imath_Y\otimes\Delta_A\otimes\imath_Y)((\imath_Y\otimes T)(\Delta_Y(y)\otimes a)(1\otimes1\otimes y'))(1\otimes1\otimes a'\otimes 1)]\\
	&\stackrel{\ref{defcomodcopar} (ii)}{=}&((\imath_A\otimes\varepsilon_Y)T\otimes\imath_A\otimes\imath_Y)[(\imath_Y\otimes\Delta_A\otimes\imath_Y)(x\otimes b\otimes w)(1\otimes1\otimes a'\otimes 1)]\\
	&=&((\imath_A\otimes\varepsilon_Y)T\otimes\imath_A\otimes\imath_Y)(x\otimes\Delta_A(b)(1\otimes a')\otimes w)\\
	&=&hb_{(1)}\varepsilon_Y(x)\otimes b_{(2)}a'\otimes w\\
	&=&(h\otimes 1\otimes 1)(\varepsilon_Y\otimes\Delta_A\otimes \imath_Y)(x\otimes b\otimes w)(1\otimes a'\otimes 1)\\
	&=&(h\otimes 1\otimes 1)(\varepsilon_Y\otimes\Delta_A\otimes \imath_Y)((\imath_Y\otimes T)(\Delta_Y(y)\otimes a)(1\otimes1\otimes y'))(1\otimes a'\otimes 1)\\
	&=&(h\otimes 1\otimes 1)((\Delta_A\otimes\imath_Y)(\varepsilon_Y\otimes T)(\Delta_Y(y)\otimes a)(1\otimes1\otimes y'))(1\otimes a'\otimes 1)\\
&=&(h\otimes 1\otimes 1)((\Delta_A\otimes\imath_Y)T((\varepsilon_Y\otimes\imath_Y)\Delta_Y(y)\otimes a)(1\otimes1\otimes y'))(1\otimes a'\otimes 1)\\
	&=&(h\otimes 1\otimes 1)((\Delta_A\otimes\imath_Y)T(y\otimes a)(1\otimes1\otimes y'))(1\otimes a'\otimes 1)\\
&=&(h\otimes 1\otimes 1)((\Delta_A(ha)\otimes y)(1\otimes1\otimes y'))(1\otimes a'\otimes 1)\\
&=&((h\otimes 1)\Delta(h)\otimes 1)(\Delta_A(a)(1\otimes a')\otimes yy'),
	\end{eqnarray*}
denoting $(\imath_Y\otimes T)(\Delta_Y(y)\otimes a)(1\otimes 1\otimes y')=x\otimes b\otimes w\in Y\otimes A\otimes Y.$ Therefore, since $\Delta_A(A)(1\otimes A)=A\otimes A$, $h\otimes h=(h\otimes 1)\Delta(h).$ The symmetric condition holds similarly.
	
	Conversely, suppose that $\varepsilon(h)=1_\Bbbk$ and $h\otimes h=(h\otimes 1)\Delta(h).$ Note that, applying $\imath\otimes\varepsilon$ in the equality $h\otimes h=(h\otimes 1)\Delta(h)$ and assuming $\varepsilon(h)=1_\Bbbk$, we have $h= h^2.$
	
	$\bullet$ For all $y\in Y,$ $y=\varepsilon(h)y=(\varepsilon\otimes \imath)\rho(y).$
	\ \
	
	$\bullet$ $(\imath_Y\otimes T)(\Delta(y)\otimes a)(1\otimes 1\otimes y')=y_{(1)}\otimes ha\otimes y_{(2)}y'\in Y\otimes A\otimes Y$ for all $y,y'\in Y$ and $a\in A.$
%	Vamos verificar a compatibilidade em $M_{0,3}^Y(Y\otimes A\otimes Y)$:
%	\begin{eqnarray*}
%		y''\overline{(\imath_Y\otimes T)(\Delta(y)\otimes a)}(y')&=&(1\otimes 1\otimes y'')[(\imath_Y\otimes T)(\Delta(y)\otimes a)(1\otimes 1\otimes y')]\\
%		&=& (1\otimes 1\otimes y'')[(\imath_Y\otimes T)(\Delta(y)(1\otimes y')\otimes a)]\\
%		&=& (1\otimes 1\otimes y''))(y_1\otimes ha\otimes y_2y')\\
%		&=& y_1\otimes ha\otimes y''y_2y'\\
%		&=& (\imath_Y\otimes\tau_{A,Y})(y_1\otimes y''y_2y'\otimes ha )\\
%		&=& (\imath_Y\otimes\tau_{A,Y})((1\otimes y'')(\Delta(y)(1\otimes y')\otimes ha )\\
%		&=& (\imath_Y\otimes\tau_{A,Y})(((1\otimes y'')(\Delta(y))(1\otimes y')\otimes ha )\\
%		&=& y_1\otimes ha\otimes y''y_2y'\\
%		&=& (y_1\otimes ha\otimes y''y_2)(1\otimes 1\otimes y')\\
%		&=&  [(\imath_Y\otimes T)(y_1\otimes y''y_2\otimes a)](1\otimes 1\otimes y')\\
%		&=& [(\imath_Y\otimes T)((1\otimes y'')\Delta(y)\otimes a)](1\otimes 1\otimes y')\\
%		&=& [(1\otimes 1\otimes y'')(\imath_Y\otimes T)(\Delta(y)\otimes a)](1\otimes 1\otimes y')\\
%		&=& \overline{\overline{(\imath_Y\otimes T)(\Delta(y)\otimes a)}}(y'')y',	
%	\end{eqnarray*}
%	para todos $y,y',y''\in Y$, $a\in A.$
	
	$\bullet$\begin{eqnarray*}
		(T\otimes \imath_Y)((\imath_Y\otimes T)(\Delta_Y(y)\otimes a))(1\otimes 1\otimes y')&\stackrel{(\ref{comppar2})}{=}&(T\otimes \imath_Y)((\imath_Y\otimes T)(\Delta_Y(y)\otimes a)(1\otimes 1\otimes y'))\\
		&=&\rho(y_{(1)})(ha\otimes 1)\otimes y_{(2)}y'\\
		&=&hha\otimes y_{(1)}\otimes y_{(2)}y'\\
		&=&ha\otimes y_{(1)}\otimes y_{(2)}y'\\
		&=&ha\otimes \Delta(y)(1\otimes y')\\
		&=&(\imath_A\otimes\Delta_Y)(\rho(y)(a\otimes 1))(1\otimes 1\otimes y')\\
		&=&(\imath_A\otimes\Delta_Y)T(y\otimes a)(1\otimes 1\otimes y'),
	\end{eqnarray*}
	for every $y,y'\in Y$ and $a\in A.$ Therefore, $(\imath_A\otimes\Delta_Y)T=(T\otimes \imath_Y)(\imath_Y\otimes T)(\Delta_Y\otimes \imath_A).$
	
	$\bullet$ $(\imath_A\otimes T)(T\otimes\imath_A)(\imath_Y\otimes\Delta_A)=((\imath_A\otimes\varepsilon_Y)T\otimes\imath_A\otimes\imath_Y)(\imath_Y\otimes(\Delta_A\otimes\imath_Y)T)(\Delta_Y\otimes\imath_A).$
	\begin{eqnarray*}
	[(\imath_A\otimes T)(T\otimes\imath_A)(y\otimes\Delta_A(a))](a'\otimes a''\otimes y')&=&(\imath_A\otimes T)((T\otimes\imath_A)(y\otimes\Delta_A(a))(a'\otimes a''\otimes y')\\
		&\stackrel{\ref{trocaLydiaparcial}}{=}&(\imath_A\otimes T)(ha_{(1)}a'\otimes y\otimes a_{(2)})(1\otimes a''\otimes y')\\
		&=&(ha_{(1)}a'\otimes ha_{(2)}\otimes y)(1\otimes a''\otimes y')\\
		&=&(h\otimes h)(a_{(1)}a'\otimes a_{(2)}\otimes y)(1\otimes a''\otimes y')\\
		&=&((h\otimes 1)\Delta(h)(\Delta(a)(a'\otimes 1))\otimes y)((1\otimes a''\otimes y')\\
		&=&(h\otimes 1)\Delta(ha)(a'\otimes a'')\otimes yy'.
	\end{eqnarray*}
On the other hand,
	\begin{eqnarray*}
		&&[((\imath_A\otimes\varepsilon_Y)T\otimes\imath_A\otimes\imath_Y)(\imath_Y\otimes(\Delta_A\otimes\imath_Y)T)(\Delta_Y(y)\otimes a)](a'\otimes a''\otimes y')\\
		&=&((\imath_A\otimes\varepsilon_Y)T\otimes\imath_A\otimes\imath_Y)[(\imath_Y\otimes(\Delta_A\otimes\imath_Y)T)(\Delta_Y(y)\otimes a)(1\otimes a''\otimes y')](a'\otimes 1^2)\\
		&=&((\imath_A\otimes\varepsilon_Y)T\otimes\imath_A\otimes\imath_Y)[(\imath_Y\otimes\Delta_A\otimes\imath_Y)((\imath_Y\otimes T)(\Delta_Y(y)\otimes a)(1\otimes 1\otimes y'))(1\otimes1\otimes a''\otimes 1)](a'\otimes 1^2)\\
		&=&((\imath_A\otimes\varepsilon_Y)T\otimes\imath_A\otimes\imath_Y)[(\imath_Y\otimes\Delta_A\otimes\imath_Y)(y_{(1)}\otimes ha\otimes y_{(2)}y')(1\otimes1\otimes a''\otimes 1)](a'\otimes 1\otimes 1)\\
		&=&(\imath_A\otimes\varepsilon_Y)T\otimes\imath_A\otimes\imath_Y)(y_{(1)}\otimes\Delta_A(ha)(1\otimes a'')\otimes y_{(2)}y')](a'\otimes 1\otimes 1)\\
		&=&(h\otimes 1)\Delta(ha)(a'\otimes a'')\otimes yy',
	\end{eqnarray*}
	for all $a, a',a''\in A$, $y, y'\in Y$. Therefore, $Y$ is a partial $A$-comodule coalgebra via $\rho.$ If we suppose $h\otimes h=\Delta(h)(h\otimes 1)$, analougosly, $Y$ is a symmetric partial $A$-comodule coalgebra via $\rho.$
	\end{proof}

\begin{exa} Considere $Y$ and $A$ regular multiplier Hopf algebras and $h\in A$ a right cointegral such that $\varepsilon(h)=1_{\mathbb{K}}$. Then, by Proposition \ref{comodcoalparcviah}, $Y$ is a partial $A$-comodule coalgebra via
	 \begin{eqnarray*}
		\rho: Y&\longrightarrow & M(A\otimes Y)\\
		y &\longmapsto & h\otimes y.
	\end{eqnarray*}	
Indeed, $(h\otimes 1)\Delta(h)=hh_{(1)}\otimes h_{(2)}=h\varepsilon(hh_{(1)})\otimes h_{(2)}=h\otimes h.$ Moreover, if $h\in A$ is a left and right cointegral such $\varepsilon(h)=1_{\mathbb{K}}$, then $Y$ is a symmetric partial $A$-comodule coalgebra via $\rho$.
	 \end{exa}

\begin{exa}\label{exemcomoparcviah}
	Consider $Y$ a regular multiplier Hopf algebra. Then, by Proposition \ref{comodcoalparcviah}, $Y$ is a symmetric partial $A_G$-comodule coalgebra via
	\begin{eqnarray*}
		\rho: Y&\longrightarrow & M(A_G\otimes Y)\\
		y &\longmapsto & h\otimes y
	\end{eqnarray*}	
	where, considering $N$ a subgroup of $G$,
	\begin{eqnarray*}
		h: G &\longrightarrow & \Bbbk\\
		s &\longmapsto &\left\{
		\begin{array}{rl}
			1 & \text{, } s\in N\\
			0 & \text{, otherwise }
		\end{array}. \right.
	\end{eqnarray*} 
	\end{exa}
\begin{rem}
The example above shows that $T$ is not always a bijection for partial comodule coalgebras.
\end{rem}

\begin{exa}
	Let $Y$ be a regular multiplier Hopf algebra and $\mathbb{H}_4=\Bbbk\langle g,x \ | \ g^2=1, x^2=0, xg=-gx \rangle$ the Sweedler Hopf algebra with coproduct given by $\Delta(g)=g\otimes g \ \mbox{and} \ \Delta(x)=x\otimes 1 + g\otimes x$. Then $Y$ is a symmetric partial $(\mathbb{H}_4\otimes A_G)$-partial comodule coalgebra via
	\begin{eqnarray*}
		\rho: Y&\longrightarrow &\mathbb{H}_4\otimes A_G\otimes Y\\
		y&\longmapsto &z\otimes h\otimes y
	\end{eqnarray*}
	where $z=\frac{1+g}{2}+\alpha gx$ (see \cite{Caenepeel}) and $h$ is given in Example \ref{exemcomoparcviah}.
\end{exa}

\begin{exa}
	Let $Y$ be a regular multiplier Hopf algebra and $T_3(q)=\Bbbk\langle g,x \ | \ g^3=1, x^3=0, xg=qgx \rangle$ the Taft algebra of order 3 with coproduct given by $\Delta(g)=g\otimes g \ \mbox{and} \ \Delta(x)=x\otimes 1 + g\otimes x$, where $q$ is a primitive $3^{th}$ root of unity.
	
	Then $Y$ is a symmetric partial $(T_3(q)\otimes A_G)$-comodule coalgebra via
	\begin{eqnarray*}
		\rho: Y&\longrightarrow &T_3(q)\otimes A_G\otimes Y\\
		y&\longmapsto &z\otimes h\otimes y
	\end{eqnarray*}
	where $z=\frac{1+g+g^2}{3}+\frac{1}{3}((q-1)\alpha gx+(q^2-1)\alpha g^2x-3q\alpha^2 gx^2)$ (see \cite{Leonardo}) and $h$ is given in Example \ref{exemcomoparcviah}.
\end{exa}

	\begin{exa}\label{exemdocorep} Let $A$ be the vector space over the complex numbers field $\mathbb{C}$ generate by the elements $\{e_pd^q ; p\in \mathbb{Z} \ \mbox{and} \ q\in\mathbb{N}\}$. Define the product in $A$ as $de_p=e_{p+1}d$\quad and\quad $e_pe_r=\delta_{p,r}e_p$,
		where $\delta$ is the Kronecker's delta. Consider $\lambda\in\mathbb{C}$ be a nonzero element and $c_{\lambda}=\sum\limits_{r\in\mathbb{Z}}\lambda^r e_r\in M(A).$
		%as follows
		%	\begin{center}
		%		$c_{\lambda}(e_pd^q)=\lambda^pe_pd^q\quad and \quad	(e_pd^q)c_{\lambda}= \lambda^{p-q}e_pd^q,$
		%	\end{center}
		%for all $e_pd^q\in A$, thus $c_{\lambda}$ 	
		The coproduct $\Delta_{\lambda}$ is defined by $\Delta_{\lambda}(e_p)= \sum_{r\in\mathbb{Z}} e_r\otimes e_{p-r}, \ \ p\in\mathbb{Z}\ \ \ \mbox{and}\ \ \  
			\Delta_{\lambda}(d)=d\otimes c_{\lambda}+1\otimes d.$
Therefore by \cite{CorepI}, $(A,\Delta_{\lambda})$ is a multiplier Hopf algebra, with antipode and counit given by 
		\begin{eqnarray*}
			S(e_p)=e_{-p}, &\ & \varepsilon(e_p)=\delta_{0,p}, \ \ p\in\mathbb{Z},\\
			S(d)=-dc_{\lambda}^{-1}, & \ & \varepsilon(d)=0.
		\end{eqnarray*}
	
Then, by Proposition \ref{comodcoalparcviah}, $Y$ is a symmetric partial  $A$-comodule coalgebra via
	\begin{eqnarray*}
		\rho: Y&\longrightarrow & A\otimes Y\\
		y &\longmapsto & e_0\otimes y.
	\end{eqnarray*}
	\end{exa}

\subsection{Induced Partial Comodule Coalgebra}
In this subsection a specific partial comodule coalgebra is constructed from a comodule coalgebra through projection. For this remember the following definition.

\begin{defi}
	Let $Z$ be an algebra and let $Y$ be a subalgebra of $Z$. A linear operator  $\pi: Z\longrightarrow Z$ is  called of \textit{projection} over $Y$, if $Im{\pi}=Y$ and $\pi(y)=y$, for all $y\in Y$. Besides that, when $\pi$ is multiplicative, we say that $\pi$ is an \textit{algebra projection}.
\end{defi}

\begin{rem} Let $(A,\Delta_A)$ be a multiplier bialgebra (or multiplier Hopf algebra) and let $B\subseteq A$ be a subalgebra. If $(B,\Delta_B)$ is a multiplier bialgebra (or multiplier Hopf algebra), we say that $B$ is a \textit{multiplier subbialgebra (or multiplier Hopf subalgebra)} of $A$, where $\Delta_B:B\longrightarrow M(B\otimes B)$ denotes the coproduct $\Delta_A$ restrict to $B.$
\end{rem}

\begin{defi}\label{defprojdealgebras}
	Let $Z$ be a regular multiplier Hopf algebra, $Y$ a regular multiplier Hopf subalgebra of $Z$ and $\pi: Z\longrightarrow Z$ an algebra projection over $Y$. We say that  $\pi$ is \textit{comultiplicative} if  $\Delta_Y\pi= (\pi\otimes\pi)\Delta_Z$, \textit i.e., 	\begin{eqnarray*}
			\Delta_Y(\pi(z))(1\otimes y)=(\pi\otimes\pi)(\Delta_Z(z)(1\otimes y)),\\
			\nonumber\Delta_Y(\pi(z))(y\otimes 1)=(\pi\otimes\pi)(\Delta_Z(z)(y\otimes 1)),
		\end{eqnarray*}
		for all $y\in Y$ and $z\in Z$.
	\end{defi}

\begin{pro} \label{comodcoinduz} Let $Z$ be an $A$-comodule coalgebra via $\rho$, $Y$ a regular multiplier Hopf subalgebra of $Z$ and $\pi$ an algebra projection of $Z$ over $Y$. If
\begin{enumerate}
	\item[(i)] $\pi$ is comultiplicative,
	\item[(ii)] 
	$(\imath_A\otimes\pi)T(\pi(z)\otimes a)=(\imath_A\otimes\pi)T(z\otimes a),$
		\item[(iii)]
		$(\imath_A\otimes\pi)T(\pi(z)\otimes a)=(\imath_A\otimes\pi)T(z_{(2)}\otimes a)\varepsilon_Y(\pi(z_{(1)}t)),$
	\end{enumerate}
for all $a\in A$, $z\in Z$ and $t\in Y$ such that $\varepsilon_Y(t)=1_{\Bbbk},$
then $Y$ is a partial $A$-comodule coalgebra via
\begin{eqnarray*}
	\beta: Y&\longrightarrow & M(A\otimes Y)\\
	y &\longmapsto & \beta(y):=(\imath_A\otimes\pi)\rho(y),
\end{eqnarray*}
where 	$(\imath_A\otimes\pi)\rho(y)(a\otimes 1)=(\imath_A\otimes\pi)(\rho(y)(a\otimes 1))$ and $(a\otimes 1)(\imath_A\otimes\pi)\rho(y)=(\imath_A\otimes\pi)((a\otimes 1)(\rho(y)),$     	
for all $y\in Y$ and $a\in A.$
 \end{pro}
\begin{proof}
	
Denote $T'(y\otimes a)=\beta(y)(a\otimes1)=(\imath_A\otimes\pi)T(y\otimes a),$ for all $a\in A$ and $y\in Y.$ Let us verify the items of Definition \ref{defcomodcopar}.
	
	$\bullet$ $(\varepsilon_A\otimes \imath_Y)\beta(y)=(\varepsilon_A\otimes \pi)\rho(y)
	=\pi((\varepsilon_A\otimes\imath_Y)\rho(y))
	\stackrel{\ref{epsiloncoacao}}{=}\pi(y)
	=y,$ for all $y\in Y.$\\

	$\bullet$ Given $y,y'\in Y$ and $a\in A$, 
	 	\begin{eqnarray*}
		((\imath_Y\otimes T')(\Delta_Y(y)\otimes a))(1\otimes1\otimes y')(y''\otimes 1^2)&=&((\imath_Y\otimes T')(\Delta_Y(y)(y''\otimes 1)\otimes a))(1\otimes 1\otimes y')\\
		&=&((\imath_Y\otimes T')(\Delta_Y(\pi(y))(y''\otimes 1)\otimes a))(1\otimes 1\otimes y')\\
		&\stackrel{(i)}{=}&(\imath_Y\otimes T')((\pi\otimes\pi)(\Delta_Z(y)(y''\otimes 1))\otimes a)(1\otimes 1\otimes y')\\
		&=&(\imath_Y\otimes T')((\pi\otimes\pi)(y_1y''\otimes y_2)\otimes a)(1\otimes 1\otimes y')\\
		&=&\pi(y_{(1)}y'')\otimes T'(\pi(y_{(2)})\otimes a)(1\otimes 1\otimes y')\\
		&\stackrel{(ii)}{=}&\pi(y_{(1)}y'')\otimes(\imath_A\otimes\pi)T(y_{(2)}\otimes a)(1\otimes 1\otimes y')\\
		&=&(\pi\otimes\imath_A\otimes\pi)(\imath_Z\otimes T)(\Delta_Z(y)(y''\otimes 1)\otimes a)(1\otimes 1\otimes y')\\
		&=&(\pi\otimes\imath_A\otimes\pi)((\imath_Z\otimes T)(\Delta_Z(y)\otimes a)(1\otimes 1\otimes y'))(y''\otimes1^2),
	\end{eqnarray*}	
	for all  $y''\in Y$. Then
	\begin{eqnarray}\label{iguald(ii)}
	((\imath_Y\otimes T')(\Delta_Y(y)\otimes a))(1\otimes 1\otimes y')=(\pi\otimes\imath_A\otimes\pi)((\imath_Z\otimes T)(\Delta_Z(y)\otimes a)(1\otimes 1\otimes y')).
	\end{eqnarray}
	By Definition \ref{comodcoalglob}, 
	$((\imath_Y\otimes T')(\Delta_Y(y)\otimes a))(1\otimes 1\otimes y')\in Y\otimes A\otimes Y,$ for all $y,y'\in Y$, $a\in A.$ Analogously, $(1\otimes 1\otimes y')((\imath_Y\otimes T')(\Delta_Y(y)\otimes a))\in Y\otimes A\otimes Y.$	
	
	$\bullet$ For all $y,y'\in Y$ and $a\in A,$
	\begin{eqnarray*}
	(\imath_A\otimes\Delta_Y)(T'(y\otimes a))(1\otimes 1\otimes y')&=&(\imath_A\otimes\Delta_Y)((\imath_A\otimes\pi)T(y\otimes a))(1\otimes1\otimes y')\\
		&\stackrel{(i)}{=}&y^{(-1)}a\otimes(\pi\otimes\pi)(\Delta_Z(y^{(0)})(1\otimes y'))\\
		&=&(\imath_A\otimes\pi\otimes\pi)((\imath_A\otimes\Delta_Z)(T(y\otimes a))(1\otimes1\otimes y'))\\
		&\stackrel{\ref{comodcoalglob}(ii)}{=}&(\imath_A\otimes\pi\otimes\pi)(((T\otimes\imath_Z)(\imath_Z\otimes T)(\Delta_Z(y)\otimes a))(1\otimes1\otimes y'))\\
		&\stackrel{\ref{comodcoalglob}(i)}{=}&(\imath_A\otimes\pi\otimes\pi)((T\otimes\imath_Z)(z\otimes b\otimes z'))\\
		&\stackrel{(ii)}{=}&T'(\pi(z)\otimes b)\otimes\pi (z')\\
		&=&(T'\otimes\imath_Y)((\pi\otimes\imath_A\otimes\pi)((\imath_Z\otimes T)(\Delta_Z(y)\otimes a))(1\otimes1\otimes y'))\\
		&\stackrel{(\ref{iguald(ii)})}{=}&(T'\otimes\imath_Y)((\imath_Y\otimes T')(\Delta_Y(y)\otimes a))(1\otimes 1\otimes y'),
	\end{eqnarray*}
where $(\imath_Z\otimes T)(\Delta_Z(y)\otimes a)(1\otimes1\otimes y'))=z\otimes b\otimes z'.$ 	
	
	$\bullet$ Let $y,y'\in Y$, $a,c\in A$, 
	\begin{eqnarray*}
		& \ &(\imath_A\otimes T')(T'\otimes\imath_A)(y\otimes\Delta_A(a))(1\otimes c\otimes y')\\
		&=&(\imath_A\otimes T')[(T'\otimes\imath_A)(y\otimes\Delta_A(a)(1\otimes c))](1\otimes 1\otimes y')\\
		&=&(\imath_A\otimes T')((\imath_A\otimes\pi)T(y\otimes a_{(1)})\otimes a_{(2)}c)(1\otimes 1\otimes y')\\
		&=&y^{(-1)}a_{(1)}\otimes T'(\pi(y^{(0)})\otimes a_{(2)}c)(1\otimes y')\\
		&=&y^{(-1)}a_{(1)}\otimes(\imath_A\otimes\pi)T(\pi(y^{(0)})\otimes a_{(2)}c)(1\otimes y')\\
		&\stackrel{(iii)}{=}&y^{(-1)}a_{(1)}\otimes(\imath_A\otimes\pi)T({y^{(0)}}_{(2)}\otimes a_{(2)}c)(1\otimes y')\varepsilon_Y(\pi({y^{(0)}}_{(1)}t))\\
		&=&(\imath_A\otimes\imath_A\otimes\pi)((\imath_A\otimes T)(y^{(-1)}a_{(1)}\otimes{y^{(0)}}_{(2)}\varepsilon_Y(\pi({y^{(0)}}_{(1)}t))\otimes a_{(2)}c))(1\otimes 1\otimes y')\\
		&=&(\imath_A\otimes\imath_A\otimes\pi)((\imath_A\otimes T)(\imath_A\otimes (\varepsilon_Y\pi)\otimes\imath_Y\otimes\imath_A)(y^{(-1)}a_{(1)}\otimes\Delta_Y(y^{(0)})(t\otimes1)\otimes a_{(2)}c))(1^2\otimes y')\\
		&=&(\imath_A\otimes\imath_A\otimes\pi)((\imath_A\otimes (\varepsilon_Y\pi)\otimes\imath_A\otimes\imath_Y)(\imath_A\otimes\imath_Y\otimes T)\\
		& \ &((\imath_A\otimes \Delta_Y)T(y\otimes a_{(1)})(1\otimes t\otimes1)\otimes a_{(2)}c))(1 \otimes 1 \otimes y')\\
		&\stackrel{\ref{comodcoalglob}(ii)}{=}&(\imath_A\otimes\imath_A\otimes\pi)((\imath_A\otimes (\varepsilon_Y\pi)\otimes\imath_A\otimes\imath_Y)(\imath_A\otimes\imath_Y\otimes T)\\
		& \ &((T\otimes\imath_Z)(\imath_Z\otimes T)(\Delta_Y(y)\otimes a_{(1)})(1\otimes t\otimes1)\otimes a_{(2)}c))(1\otimes 1\otimes y')\\
		&=&(\imath_A\otimes\imath_A\otimes\pi)(\imath_A\otimes (\varepsilon_Y\pi)\otimes\imath_A\otimes\imath_Y)\\
		& \ &[((\imath_A\otimes\imath_Y\otimes T)((T\otimes\imath_Z\otimes\imath_A)(\imath_Y\otimes T\otimes\imath_A)(\Delta_Y(y)\otimes a_{(1)}\otimes a_{(2)}c)))(1\otimes t\otimes1\otimes 1)](1^2\otimes y')\\
		&=&(\imath_A\otimes\imath_A\otimes\pi)(\imath_A\otimes (\varepsilon_Y\pi)\otimes\imath_A\otimes\imath_Y)\\
		& \ & [((T\otimes\imath_A\otimes\imath_Y)(\imath_Y\otimes\imath_A\otimes T)(\imath_Y\otimes T\otimes\imath_A)(\Delta_Y(y)\otimes \Delta_A(a)(1\otimes c)))(1\otimes t\otimes1\otimes 1)](1^2 \otimes y')\\
		&=&(\imath_A\otimes\imath_A\otimes\pi)(\imath_A\otimes (\varepsilon_Y\pi)\otimes\imath_A\otimes\imath_Y)\\
		& \  &[(T\otimes\imath_A\otimes\imath_Y)((\imath_Y\otimes(\imath_A\otimes T)(T\otimes\imath_A)(\imath_Y\otimes\Delta_A))(\Delta_Y(y)\otimes a))(1\otimes1\otimes c\otimes1)](1\otimes 1\otimes y')\\
		&\stackrel{\ref{paraequiv}(ii)}{=}&(\imath_A\otimes\imath_A\otimes\pi)(\imath_A\otimes(\varepsilon_Y\pi)\otimes\imath_A\otimes\imath_Y)\\
		& \ &[(T\otimes\imath_A\otimes\imath_Y)(\imath_Y\otimes(\Delta_A\otimes\imath_Y)T)(\Delta_Y(y)\otimes a)(1\otimes1\otimes c\otimes1)](1\otimes 1\otimes y')\\
		&=&((\imath_A\otimes\imath_A\otimes\pi)((\imath_A\otimes\varepsilon_Y)T'\otimes\imath_A\otimes\imath_Y)(\imath_Y\otimes(\Delta_A\otimes\imath_Y)T)(\Delta_Y(y)\otimes a))(1\otimes1\otimes c\otimes y')\\
		&=&(((\imath_A\otimes\varepsilon_Y)T'\otimes\imath_A\otimes\imath_Y)(\imath_Y\otimes\Delta_A\otimes\imath_Y)(\imath_Y\otimes(\imath_A\otimes \pi)T)(\Delta_Y(y)\otimes a))(1\otimes1\otimes c\otimes y')\\
		&=&(((\imath_A\otimes\varepsilon_Y)T'\otimes\imath_A\otimes\imath_Y)(\imath_Y\otimes(\Delta_A\otimes\imath_Y)T')(\Delta_Y(y)\otimes a))(1\otimes c\otimes y'),
	\end{eqnarray*}
	for all $y'\in Y$, $c\in A.$ Therefore, $Y$ is a partial $A$-comodule coalgebra via $\beta.$
\end{proof}
In this case, we say that  $Y$ is an \textit{induced partial $A$-comodule coalgebra.}

\begin{exa}
	Consider $N$ a subgroup of a group $G$ such that there is an idempotent element $t\in N$. Thus the group algebra $\Bbbk N$ is an $A_G$-comodule coalgebra via
	\begin{eqnarray*}
		\rho: \Bbbk N & \longrightarrow &  M(A_G \otimes \Bbbk N) \\
		h & \longmapsto & \displaystyle \sum_{g \in N}  \delta_{g} \otimes hg.
	\end{eqnarray*}
	Now, define
	\begin{eqnarray*}
		\pi: \Bbbk N &\rightarrow& \Bbbk N\\
		h &\mapsto& \left\{
		\begin{array}{rl}
			h, & \text{if $h \in Y$ },\\
			0, & \text{ otherwise, }
		\end{array} \right.
	\end{eqnarray*}
	where $Y = \Bbbk \langle t, \ t^{2}=1_{G}\rangle$. Then,  $\pi$ satisfies Proposition \ref{comodcoinduz}. Therefore, $Y$ is an \textit{induced partial $A_G$-comodule coalgebra via}
	\begin{eqnarray*}
		\beta: Y & \longrightarrow &  M(A_G \otimes Y) \\
		y & \longmapsto & (\imath_{A_G}\otimes\pi)\rho(y).
	\end{eqnarray*} Moreover, the induced coaction constructed is not global. 		
\end{exa}

\section{Partial Smash Coproduct}
\quad \ The construction of the smash coproduct associated with a comodule coalgebra, as a dual notion of the smash product, was initially formulated by R. Molnar in \cite{Molnar}. In \cite{Lydia} L. Delvaux generalized this notion to the context of multiplier Hopf algebras. In the case of Hopf algebras, E. Batista and J. Vercruysse constructed, in \cite{Batista}, the partial smash coproduct associated to a partial comodule coalgebra. In this section we will extend the results presented by L. Delvaux to the partial context, also generalizing the classic case of Hopf algebras.

Initially, let $Y$ be a partial $A$-comodule coalgebra. Based on \cite{Lydia}, we can define two linear maps $\overline{T}_1$, $\overline{T}_2$ on $(Y\otimes A)\otimes(Y\otimes A)$: for all $y,y'\in Y$ and $a,a'\in A$,
\begin{eqnarray*}
	\bullet\ \ \overline{T}_1((y\otimes a)\otimes (y'\otimes a'))=((\imath_Y\otimes T)(\Delta_Y(y)\otimes a_{(1)}))(1\otimes1\otimes y')\otimes a_{(2)}a'
\end{eqnarray*}
\begin{eqnarray*}
	\bullet\ \ \overline{T}_2((y'\otimes a')\otimes (y\otimes a))= y'y_{(1)}\otimes(a'\otimes1\otimes1)((T\otimes\imath_A)(y_{(2)}\otimes\Delta_A(a))).
\end{eqnarray*}

From Definition \ref{defcomodcopar} and Remark \ref{trocaLydiaparcial}, it follows that $\overline{T}_1$ and $\overline{T}_2$ are well defined, respectively. We will use the maps  $\overline{T}_1$ and $\overline{T}_2$  to define a comultiplication $\overline{\Delta}$ on $Y\otimes A.$

\begin{pro}\label{defDelta} Let $Y$ be a partial $A$-comodule coalgebra. Given $y\in Y$, $a\in A$, $\overline{\Delta}(y\otimes a)$ defines a multiplier in $M((Y\otimes A)\otimes(Y\otimes A))$ in the following way:
	\begin{eqnarray}\label{deltasmash}
		&&\hspace{-2.5cm}\overline{\Delta}(y\otimes a)((y'\otimes a')\otimes (y''\otimes a''))= \overline{T}_1((y\otimes a)\otimes (y''\otimes a''))((y'\otimes a')\otimes(1\otimes1));\nonumber\\
		&&\hspace{-2.5cm}((y'\otimes a')\otimes (y''\otimes a''))\overline{\Delta}(y\otimes a)= ((1\otimes1)\otimes(y''\otimes a''))\overline{T}_2((y'\otimes a')\otimes (y\otimes a)),
	\end{eqnarray}
	for all $y',y''\in Y$ and $a',a''\in A.$ Moreover, $\overline{\Delta}$ is coassociative.	
\end{pro}
\begin{proof}
	Using product associativity in $(Y\otimes A)\otimes(Y\otimes A)$ is straightforward to verify $\overline{\Delta}(y\otimes a)\in M((Y\otimes A)\otimes(Y\otimes A)).$ 
		
	We will show the coassociativity of $\overline{\Delta}$, for every $y,y',y''\in Y$, $a,a',a''\in A.$ On the one hand,
	\begin{eqnarray*}
		& \ &(\imath\otimes\overline{\Delta})((y'\otimes a'\otimes 1\otimes 1)\overline{\Delta}(y\otimes a))(1\otimes 1\otimes 1\otimes 1\otimes y''\otimes a'')\\
		&=&(\imath\otimes\overline{\Delta})(y'y_{(1)}\otimes(a'\otimes 1\otimes1)((T\otimes\imath_A)(y_{(2)}\otimes\Delta_A(a)))(1\otimes 1\otimes 1\otimes 1\otimes y''\otimes a'')\\
		&\stackrel{\ref{trocaLydiaparcial}}{=}&(\imath\otimes\overline{\Delta})(y'y_{(1)}\otimes a'{y_{(2)}}^{(-1)}a_{(1)}\otimes {y_{(2)}}^{(0)}\otimes a_{(2)})(1\otimes 1\otimes 1\otimes 1\otimes y''\otimes a'')\\
		&=&(y'y_{(1)}\otimes a'{y_{(2)}}^{(-1)}a_{(1)}\otimes\overline{\Delta}({y_{(2)}}^{(0)}\otimes a_{(2)})(1\otimes 1\otimes y''\otimes a'')\\
		&=&y'y_{(1)}\otimes a'{y_{(2)}}^{(-1)}a_{(1)}\otimes((\imath_Y\otimes T)\Delta_Y({y_{(2)}}^{(0)})\otimes a_{(2)})(1\otimes 1\otimes y'')\otimes a_{(3)}a''\\
		&=&[(\imath_{Y\otimes A\otimes Y}\otimes T\otimes\imath_A)(\imath_Y\otimes\imath_A\otimes\Delta_Y\otimes\imath_{A\otimes A})(y'y_{(1)}\otimes (a'\otimes 1^2)\\
		& \ &((T\otimes\imath_A)({y_{(2)}}\otimes\Delta_A(a_{(1)}))\otimes a_{(2)}a''))](1 \otimes 1 \otimes 1 \otimes 1 \otimes y''\otimes1)\\
		&=&(1\otimes a'\otimes 1^4)[(\imath_{Y\otimes A\otimes Y}\otimes T\otimes\imath_A)(\imath_Y\otimes(\imath_A\otimes\Delta_Y)T\otimes\imath_{A\otimes A})\\
		& \ &(y'y_{(1)}\otimes y_{(2)}\otimes\Delta(a_{(1)})\otimes a_{(2)}a'')](1 \otimes 1 \otimes 1 \otimes 1\otimes y''\otimes1)\\
		&\stackrel{\ref{defcomodcopar}(iii)}{=}&(1\otimes a'\otimes 1^4)[(\imath_{Y\otimes A\otimes Y}\otimes T\otimes\imath_A)(\imath_Y\otimes(T\otimes\imath_Y)(\imath_Y\otimes T)(\Delta_Y\otimes\imath_A)\otimes\imath_{A\otimes A})\\
		& \ &(y'y_{(1)}\otimes y_{(2)}\otimes\Delta(a_{(1)})\otimes a_{(2)}a'')](1^4\otimes y''\otimes1)\\
		&=&(1\otimes a'\otimes 1^4)[(\imath_Y\otimes T\otimes\imath_{A\otimes  Y\otimes A})(\imath_Y\otimes\imath_Y \otimes \imath_A\otimes T\otimes\imath_A)\\
		& \ &(\imath_Y\otimes(\imath_Y\otimes T)(\Delta_Y\otimes\imath_A)\otimes\imath_{A\otimes A})(y'y_{(1)}\otimes y_{(2)}\otimes\Delta_A(a_{(1)})\otimes a_{(2)}a'')](1^4\otimes y''\otimes1)\\
		&=&(1\otimes a'\otimes 1^4)[(\imath_Y\otimes T\otimes\imath_{A\otimes Y\otimes A})(\imath_Y\otimes\imath_Y \otimes (\imath_A\otimes T)(T\otimes\imath_A)\otimes\imath_A)\\
		& \ &(\imath_Y\otimes\Delta_Y\otimes\imath_A\otimes\imath_A\otimes\imath_A)((y'\otimes1)\Delta_Y(y)\otimes\Delta(a_{(1)})\otimes a_{(2)}a'')](1^4\otimes y''\otimes1)\\.
		&=&(1\otimes a'\otimes 1^4)[(\imath_Y\otimes T\otimes\imath_{A\otimes Y\otimes A})(\imath_{Y\otimes Y} \otimes (\imath_A\otimes T)(T\otimes\imath_A)(\imath_Y\otimes\Delta_A)\otimes\imath_A)\\
		& \ &(y'y_{(1)}\otimes\Delta_Y(y_{(2)})\otimes a_{(1)}\otimes a_{(2)}a'')](1^4\otimes y''\otimes1).
	\end{eqnarray*}	
	
	On the other hand, denoting $\overline{T}_1(y\otimes a\otimes y''\otimes a'')=x\otimes e\otimes z\otimes e'\in Y\otimes A\otimes Y\otimes A,$
	\begin{eqnarray*}
		& \ &(y'\otimes a'\otimes1\otimes 1\otimes 1\otimes 1)(\overline{\Delta}\otimes\imath)(\overline{\Delta}(y\otimes a)(1\otimes1\otimes y''\otimes a''))\\
		&=&(y'\otimes  a'\otimes1\otimes1\otimes1\otimes1)(\overline{\Delta}\otimes\imath)(\overline{T}_1(y\otimes a\otimes y''\otimes a''))\\
		&=&(y'\otimes a'\otimes1\otimes1)\overline{\Delta}(x\otimes e)\otimes z\otimes e'\\
		&=&y'x_{(1)}\otimes (a'\otimes1\otimes1)((T\otimes\imath_A)(x_{(2)}\otimes\Delta_A(e)))\otimes z\otimes e'\\
		&=&(y'\otimes a'\otimes1^4)[(\imath_Y\otimes T\otimes\imath_A\otimes\imath_Y\otimes\imath_A)(\Delta_Y\otimes\Delta_A\otimes\imath_Y\otimes\imath_A)(x\otimes e\otimes z\otimes e')]\\
		&=&(y'\otimes a'\otimes1^4)[(\imath_Y\otimes T\otimes\imath_A\otimes\imath_Y\otimes\imath_A)(\Delta_Y\otimes\Delta_A\otimes\imath_Y\otimes\imath_A)\\
		& \ &((\imath_Y\otimes T)(\Delta_Y(y)\otimes a_{(1)})(1\otimes1\otimes y'')\otimes a_{(2)}a'')]\\
		&=&(y'\otimes a'\otimes1^4)[(\imath_Y\otimes T\otimes\imath_{A\otimes Y\otimes A})(\imath_{Y\otimes Y}\otimes\Delta_A\otimes\imath_{Y\otimes A})\\
		& \ &(\Delta_Y\otimes T\otimes\imath_A)(\Delta_Y(y)\otimes a_{(1)}\otimes a_{(2)}a'')](1^4\otimes y''\otimes 1)\\
		&=&(y'\otimes a'\otimes1^4)[(\imath_Y\otimes T\otimes\imath_{A\otimes Y\otimes A})(\imath_Y\otimes\imath_Y\otimes\Delta_A\otimes\imath_Y\otimes\imath_A)(\imath_Y\otimes\imath_Y\otimes T\otimes\imath_A)\\
		& \ &((\Delta_Y\otimes\imath_Y)\Delta_Y\otimes\imath_A\otimes\imath_A)(y\otimes a_{(1)}\otimes a_{(2)}a'')](1^4\otimes y''\otimes 1)\\
		&=&(y'\otimes a'\otimes1^4)[(\imath_Y\otimes T\otimes\imath_{A\otimes Y\otimes A})(\imath_Y\otimes\imath_Y\otimes\Delta_A\otimes\imath_Y\otimes\imath_A)(\imath_Y\otimes\imath_Y\otimes T\otimes\imath_A)\\
		& \ &((\imath_Y\otimes\Delta_Y)\Delta_Y\otimes\imath_A\otimes\imath_A)(y\otimes a_{(1)}\otimes a_{(2)}a'')](1^4\otimes y''\otimes 1)\\
		&=&(y'\otimes a'\otimes1^4)[(\imath_Y\otimes T\otimes\imath_{A\otimes Y\otimes A})(\imath_Y\otimes\imath_Y\otimes\Delta_A\otimes\imath_Y\otimes\imath_A)(\imath_Y\otimes(\imath_Y\otimes T)\\
		& \ &(\Delta_Y\otimes\imath_A)\otimes\imath_A)(\Delta_Y(y)\otimes a_{(1)}\otimes a_{(2)}a'')](1^4\otimes y''\otimes 1)\\	
		&=&(1\otimes a'\otimes1^4)[(\imath_Y\otimes (T\otimes\imath_A)(\imath_Y\otimes\Delta_A)\otimes\imath_Y\otimes\imath_A)(y'y_{(1)}\otimes(\imath_Y\otimes T)\\
		& \ &(\Delta_Y\otimes\imath_A)(y_{(2)}\otimes a_{(1)})\otimes a_{(2)}a'')](1^4\otimes y''\otimes 1)\\	
		&\stackrel{(\ast)}{=}&(1\otimes a'\otimes1^4)\{(\imath_Y\otimes [((\imath_A\otimes\imath_Y\otimes\varepsilon_Y)((\imath_A\otimes\Delta_Y)T)\otimes\imath_A)(\imath_Y\otimes\Delta_A)]\otimes\imath_Y\otimes\imath_A)\\
		& \ &(y'y_{(1)}\otimes(\imath_Y\otimes T)(\Delta_Y\otimes\imath_A)(y_{(2)}\otimes a_{(1)})\otimes a_{(2)}a'')\}(1^4\otimes y''\otimes 1)\\
		&\stackrel{\ref{defcomodcopar}(iii)}{=}&(1\otimes a'\otimes1^4)\{(\imath_Y\otimes [((\imath_A\otimes\imath_Y\otimes\varepsilon_Y)(T\otimes\imath_Y)(\imath_Y\otimes T)(\Delta_Y\otimes\imath_A)\otimes\imath_A)(\imath_Y\otimes\Delta_A)]\otimes\imath_Y\otimes\imath_A)\\
		& \ &(y'y_{(1)}\otimes(\imath_Y\otimes T)(\Delta_Y(y_{(2)})\otimes a_{(1)})\otimes a_{(2)}a'')\}(1^4\otimes y''\otimes 1)\\
		&=&(1\otimes a'\otimes1^4)[(\imath_Y\otimes (\imath_A\otimes\imath_Y\otimes\varepsilon_Y)(T\otimes\imath_Y)(\imath_Y\otimes T)\otimes\imath_A\otimes\imath_Y\otimes\imath_A)\\
		& \ &(\imath_Y\otimes\Delta_Y\otimes\Delta_A\otimes\imath_Y\otimes\imath_A)(y'y_{(1)}\otimes(\imath_Y\otimes T)(\Delta_Y(y_{(2)})\otimes a_{(1)})\otimes a_{(2)}a'')](1^4\otimes y''\otimes 1)\\
		&=&(1\otimes a'\otimes1^4)[(\imath_Y\otimes (\imath_A\otimes\imath_Y\otimes\varepsilon_Y)(T\otimes\imath_Y)(\imath_Y\otimes T\otimes\imath_A\otimes\imath_Y\otimes\imath_A))\\
		& \ &(\imath_Y\otimes\imath_Y\otimes\imath_Y\otimes\Delta_A\otimes\imath_Y\otimes\imath_A)(y'y_{(1)}\otimes(\Delta_Y\otimes T)(\Delta_Y(y_{(2)})\otimes a_{(1)})\otimes a_{(2)}a'')](1^4\otimes y''\otimes 1)\\
		&=&(1\otimes a'\otimes1^4)[(\imath_Y\otimes (\imath_A\otimes\imath_Y\otimes\varepsilon_Y)(T\otimes\imath_Y)(\imath_Y\otimes T)\otimes\imath_A\otimes\imath_Y\otimes\imath_A)\\
		& \ &(\imath_Y\otimes\imath_Y\otimes\imath_Y\otimes(\Delta_A\otimes\imath_Y)T\otimes\imath_A)(y'y_{(1)}\otimes(\Delta_Y\otimes\imath_Y)\Delta_Y(y_{(2)})\otimes a_{(1)}\otimes a_{(2)}a'')](1^4\otimes y''\otimes 1)\\
		&=&(1\otimes a'\otimes1^4)[(\imath_Y\otimes (\imath_A\otimes\imath_Y\otimes\varepsilon_Y)(T\otimes\imath_Y)(\imath_Y\otimes T)\otimes\imath_A\otimes\imath_Y\otimes\imath_A)\\
		& \ &(\imath_Y\otimes\imath_Y\otimes\imath_Y\otimes(\Delta_A\otimes\imath_Y)T\otimes\imath_A)(y'y_{(1)}\otimes(\imath_Y\otimes\Delta_Y)\Delta_Y(y_{(2)})\otimes a_{(1)}\otimes a_{(2)}a'')](1^4\otimes y''\otimes 1)\\
		&=&(1\otimes a'\otimes1^4)[(\imath_Y\otimes T (\imath_Y\otimes(\imath_A\otimes\varepsilon_Y)T)\otimes\imath_A\otimes\imath_Y\otimes \imath_A)(\imath_Y\otimes\imath_Y\otimes\imath_Y\otimes(\Delta_A\otimes\imath_Y)T\otimes\imath_A)\\
		& \ &(y'y_{(1)}\otimes(\imath_Y\otimes\Delta_Y\otimes\imath_A)(\Delta_Y(y_{(2)})\otimes a_{(1)})\otimes a_{(2)}a'')](1^4\otimes y''\otimes 1)\\
		&=&(1\otimes a'\otimes1^4)[(\imath_Y\otimes T \otimes\imath_A\otimes\imath_Y\otimes\imath_A)(\imath_Y\otimes\imath_Y\otimes(\imath_A\otimes\varepsilon_Y)T\otimes\imath_A\otimes\imath_Y\otimes \imath_A)\\
		& \ &(\imath_Y\otimes\imath_Y\otimes\imath_Y\otimes(\Delta_A\otimes\imath_Y)T\otimes\imath_A)
		(\imath_{Y\otimes Y}\otimes\Delta_Y\otimes\imath_{A \otimes A})(y'y_{(1)}\otimes\Delta_Y(y_{(2)})\otimes a_{(1)}\otimes a_{(2)}a'')](1^4\otimes y''\otimes 1)\\
		&=&(1\otimes a'\otimes1^4)[(\imath_Y\otimes T \otimes\imath_A\otimes\imath_Y\otimes\imath_A)(\imath_Y\otimes\imath_Y\otimes(((\imath_A\otimes\varepsilon_Y)T\otimes\imath_A\otimes\imath_Y)\\
		& \ &(\imath_Y\otimes(\Delta_A\otimes\imath_Y)T)(\Delta_Y\otimes\imath_A))\otimes\imath_A)(y'y_{(1)}\otimes\Delta_Y(y_{(2)})\otimes a_{(1)}\otimes a_{(2)}a'')](1^4\otimes y''\otimes 1)\\
		&\stackrel{\ref{defcomodcopar}(iv)}{=}&(1\otimes a'\otimes1^4)[(\imath_Y\otimes T \otimes\imath_A\otimes\imath_Y\otimes\imath_A)(\imath_Y\otimes\imath_Y\otimes(\imath_A\otimes T)(T\otimes\imath_A)(\imath_Y\otimes\Delta_A)\otimes\imath_A)\\
		& \ &(y'y_{(1)}\otimes\Delta_Y(y_{(2)})\otimes a_{(1)}\otimes a_{(2)}a'')](1^4\otimes y''\otimes 1).  
	\end{eqnarray*}
In $(\ast)$,\quad we used   $T\otimes\imath_A=(\imath_A\otimes\imath_Y\otimes\imath_A)(T\otimes\imath_A)=(\imath_A\otimes\imath_Y\otimes\varepsilon_Y)(\imath_A\otimes\Delta_Y)T\otimes\imath_A.$	
	\end{proof}
%\begin{eqnarray*}
%(1\otimes a'\otimes 1^4)\circ[(\imath_Y\otimes T\otimes\imath_A\otimes\imath_Y\otimes\imath_A)\circ(\imath_Y\otimes\imath_Y \otimes (\imath_A\otimes T)(T\otimes\imath_A)(\imath_Y\otimes\Delta_A)\otimes\imath_A)\\
%\circ(y'y_1\otimes\Delta_Y(y_2)\otimes a_1\otimes a_2a'')]
%\circ(1^4\otimes y''\otimes1)=\\
%=(1\otimes a'\otimes1^4)\circ[(\imath_Y\otimes T \otimes\imath_A\otimes\imath_Y\otimes\imath_A)\circ(\imath_Y\otimes\imath_Y\otimes(\imath_A\otimes T)(T\otimes\imath_A)(\imath_Y\otimes\Delta_A)\otimes\imath_A)\\
%\circ(y'y_1\otimes\Delta_Y(y_2)\otimes a_1\otimes a_2a'')]
%\circ(1^4\otimes y''\otimes 1),  
%\end{eqnarray*}

\begin{rem} \label{peixe_regular}Let $Y$ be a partial  $A$-comodule coalgebra. Given $y,y'\in Y$, $a,a'\in A$, 
	\begin{enumerate}
		\item[(1)]
	\begin{eqnarray*}
	\overline{\Delta}(y\otimes a)((y'\otimes a')\otimes (1\otimes 1))&=& y_{(1)}y'\otimes (T\otimes\imath_A)(y_{(2)}\otimes\Delta_A(a))(a'\otimes 1 \otimes 1)\\
	&\stackrel{\ref {trocaLydiaparcial}}{=}& y_{(1)}y'\otimes y_{(2)}^{(-1)}a_{(1)}a'\otimes y_{(2)}^{(0)}\otimes a_{(2)}, 
	\end{eqnarray*}
\item[(2)]	\begin{eqnarray*}
	((1\otimes 1)\otimes (y'\otimes a'))\overline{\Delta}(y\otimes a)&=& (1\otimes 1\otimes y')(\imath_Y\otimes T)(\Delta_Y(y)\otimes a_{(1)})\otimes a'a_{(2)},
\end{eqnarray*}
\end{enumerate}
then $\overline{\Delta}(y\otimes a)((y'\otimes a')\otimes (1\otimes 1))$ and 
	$((1\otimes 1)\otimes (y'\otimes a'))\overline{\Delta}(y\otimes a)$
belongs to $Y\otimes A\otimes Y\otimes A.$ 
\end{rem}

It still remains to show that $\overline{\Delta}$ is a homomorphism. First of all, we will remember the definition of partial comodule algebra presented in \cite{Grasiela}.
\begin{defi}\cite{Grasiela} We call $Y$ a \emph{partial $A$-comodule algebra} if $\rho: Y \longrightarrow M(A\otimes Y)$ is an injective homomorphism  and $E\in M(A\otimes Y)$ is an idempotent such that $(A\otimes 1)E\subseteq A\otimes M(Y)$ and $E(A\otimes 1)\subseteq A\otimes M(Y)$, satisfying
	\begin{enumerate}
		\item[(i)] $\rho(Y)(A\otimes 1)\subseteq E(A\otimes Y)$ and $(A\otimes 1)\rho(Y)\subseteq(A\otimes Y)E$
		\item[(ii)] $(\imath_A\otimes \rho)(\rho(y)) = (1\otimes E)(\Delta_A\otimes \imath_Y)(\rho(y))$,
	\end{enumerate}
	for all $y\in Y$. In this case, $\rho$ is called a partial coaction of $A$ on $Y$. We say that the coaction $\rho$ is \emph{symmetric} if, besides the above conditions,  $\rho$ also satisfies
	\begin{enumerate}
		\item[(iii)] $( \imath_A \otimes\rho)(\rho(y))= (\Delta_A\otimes \imath_Y)(\rho(y))(1\otimes E)$,  for all $y\in Y$.
	\end{enumerate}
	\label{def_comoalgparcarcial_multip} 	
\end{defi}

\begin{rem}\label{rhoigual_rhoe} When $Y$ is a partial $A$-comodule algebra,  $E\rho(y)=\rho(y) \ \ \mbox{and} \ \ \rho(y)E=\rho(y),$
	for all $y\in Y$.
		\end{rem}

\begin{defi}\label{defbialgebra}Let $Y$ and $A$ be two regular multiplier Hopf algebras. We say that $Y$ is a \textit{partial} $A$-\textit{comodule bialgebra} if there exists a linear map $\rho:Y\longmapsto M(A\otimes Y)$ defining a structure of partial comodule algebra and partial comodule coalgebra.
\end{defi}
We say that $Y$ is a \textit{symmetric partial} $A$-\textit{comodule bialgebra} when both structures satisfy its symmetric condition respectively. 	

\begin{rem} 	
	\begin{itemize}
		\item [(1)] If $Y$ is a symmetric partial  $A$-comodule algebra, then $\rho:{Y}\longmapsto M(A\otimes Y)$ is a homomorphism that can be extended uniquely to $M(Y)$ and $\rho(1_{M(Y)})=E$, by Proposition 3.2.3 of \cite{Grasiela}.
		\item [(2)] If $A$ is commutative, then the linear map $T$ is a homomorphism.	
	\end{itemize}
	
\end{rem}
An important consequence of item (2) of the above remark is the following result.
\begin{pro}\label{Propexten}Let $Y$ be a symmetric partial $A$-comodule bialgebra such that $A$ is commutative. Then the linear map $T$ can be extended uniquely to $M(Y\otimes A)$. Moreover, if we denote this extension by $\tilde{T}$, then $\tilde{T}(1_{M(Y\otimes A)})=E.$
\end{pro}
\begin{proof}
	By Proposition 3.14 of \cite{Grasiela},  $\rho(Y)(A\otimes 1)=E(A\otimes Y)$ and $(A\otimes 1)\rho(Y)=(A\otimes Y)E$, where $E$ is given by the structure of partial $A$-comodule algebra on $Y.$ Define
	\begin{eqnarray*}
		\tilde{T}: M(Y\otimes A) &\longrightarrow & M(A\otimes Y)\\
		m &\longmapsto & \tilde{T}(m)=(\overline{\tilde{T}(m)},\overline{\overline{\tilde{T}(m)}})
	\end{eqnarray*}
	such that, for all $y\in Y$, $b\in A$, $\overline{\tilde{T}(m)}(b\otimes y)=T(m(x\otimes a))$ and $\overline{\overline{\tilde{T}(m)}}(b\otimes y)=T((z\otimes c)m),$
	where $E(b\otimes y)=\rho(x)(a\otimes 1)$ and $(b\otimes y)E=(c\otimes 1)\rho(z).$		
	
	$\bullet$\ \ Since $A$ is commutative, $\rho(x)(a\otimes 1)=(a\otimes 1)\rho(x)$, for all $a\in A$ and $x\in Y,$ and $E$ is idempotent, it follows that $\tilde{T}$ is well defined.
	
			$\bullet$\ \  $\tilde{T}$ extends $T$: denoting $E(c'\otimes y')=\rho(x)(a\otimes 1)$, we obtain
	\begin{eqnarray*}
		\overline{\tilde{T}(y\otimes c)}(c'\otimes y')&=& T((y\otimes c)(x\otimes a))\\
		&=&\rho(yx)(ca\otimes 1)\\
		&=&\rho(y)\rho(x)(a\otimes 1)(c\otimes 1)\\
		&=&\rho(y)E(c'\otimes y')(c\otimes 1)\\
		&\stackrel{\ref{rhoigual_rhoe}}{=}&\rho(y)(c'c\otimes y')\\
		&=&T(y\otimes c)(c'\otimes y')\\
		&=&\overline{T(y\otimes c)}(c'\otimes y'),
	\end{eqnarray*}
	for all $y'\in Y$, $c'\in A.$ Therefore $\overline{\tilde{T}(y\otimes c)}=\overline{T(y\otimes c)}$. Similarly, one can show that  $\overline{\overline{\tilde{T}(y\otimes c)}}=\overline{\overline{{T}(y\otimes c)}}$. Then, $\tilde{T}(y\otimes c)=T(y\otimes c)$, for all $y\in Y$, $c\in A.$
	
	$\bullet$\ \ $\tilde{T}(1_{M(Y\otimes A)})=E$: denoting $E(b\otimes y)=\rho(x)(a\otimes 1),$
	\begin{eqnarray*}
		\overline{\tilde{T}(1_{M(Y\otimes A)})}(b\otimes y)		=T(x\otimes a)=E(b\otimes y),
	\end{eqnarray*}
	for all $b\in A$, $y\in Y.$ Analogously, $\overline{\overline{\tilde{T}(1_{M(Y\otimes A)})}}=\overline{\overline{E}}.$
	
	$\bullet$\ \ $\tilde{T}$ is a homomorphism: consider $m,n\in M(Y\otimes A)$, denoting $E(b\otimes y)=\rho(x)(a\otimes 1),$ $n(x\otimes a)=z\otimes c$ and $\rho(z)(c\otimes 1)=w\otimes d$ we have
	 
		\begin{eqnarray*}
			\overline{\tilde{T}(m){\tilde{T}(n)}}(b\otimes y)			&=&\overline{\tilde{T}(m)}(T(n(x\otimes a)))\\
			&=&\overline{\tilde{T}(m)}(T(z\otimes c))\\
			&=&\overline{\tilde{T}(m)}(\rho(z)(c\otimes 1))\\
			&=&\overline{\tilde{T}(m)}(w\otimes d)\\
			&=&T(m(z\otimes c))\\
	&=&T(mn(x\otimes a))\\
	&=&	\overline{\tilde{T}(mn)}(b\otimes y),		
	\end{eqnarray*}
	for all $b\in A$, $y\in Y,$ using that $E(w\otimes d)=E\rho(z)(c\otimes 1)=\rho(z)(c\otimes 1).$ Thus $\overline{\tilde{T}(mn)}=\overline{\tilde{T}(m){\tilde{T}(n)}}.$ Similarly we show that the second component is a homomorphism.
	
	$\bullet$\ \ $\tilde{T}$ is unique: suppose that $S$ is a homomorphism that also extends $T$ and $S(1_{M(Y\otimes A)})=E$. Thus, if $E(b\otimes y)=T(x\otimes a),$ then
	\begin{eqnarray*}
		\overline{\tilde{T}(m)}(b\otimes y)&=&T(m(x\otimes a))\\
		&=&S({m}(x\otimes a))\\
		&=&S({m})S(x\otimes a)\\
		&=&S({m})T(x\otimes a)\\
		&=&S({m})E(b\otimes y)\\
		&=&S({m})S(1_{M(Y\otimes A)})(b\otimes y)\\
		&=&S({m})(b\otimes y)\\
		&=&\overline{S({m})}(b\otimes y),	
	\end{eqnarray*}
	for all $b\in A$, $y\in Y$. Then, $\overline{\tilde{T}(m)}=\overline{S({m})}.$ The same argument applies for the second component. 
	\end{proof}

\begin{rem} Under the conditions of Proposition \ref{Propexten}, we have $\tilde{T}(1_{M(Y)}\otimes c)=(c\otimes 1)E,$ for all $c\in A.$
	\begin{eqnarray*}
		\overline{\tilde{T}(1_{M(Y)}\otimes c)}(b\otimes y)&=&T((1_{M(Y)}\otimes c)(x\otimes a))\\
		&=&T(x\otimes ca)\\
		&=&\rho(x)(ca\otimes 1)\\
		&=&(c\otimes 1)\rho(x)(a\otimes 1)\\
		&=&(c\otimes 1)E(b\otimes y)\\
		&=&\overline{(c\otimes 1)E}(b\otimes y),
	\end{eqnarray*}
	for every $b\in A$, $y\in Y.$ An analogous argument holds for the second component.\end{rem}

\begin{thm}\label{thmultiplierbialg}
	Let $Y$ be a symmetric partial $A$-comodule bialgebra such that $A$ is commutative. Then $Y\otimes A$ is a multiplier bialgebra.
\end{thm}
\begin{proof} Regarding Proposition \ref{defDelta} and Remark  \ref{peixe_regular}, it is enough to show that $\overline{\Delta}$ is a homomorphism considering the usual product in $Y\otimes A$. Indeed,
	\begin{eqnarray*}
		& \ &(y''\otimes a''\otimes 1\otimes 1))\overline{\Delta}(yy'\otimes aa')\\
		&=&\overline{T}_2(y''\otimes a''\otimes yy'\otimes aa')\\
		&=&y''(yy')_{(1)}\otimes (a''\otimes 1\otimes 1)((T\otimes \imath_A)((yy')_{(2)}\otimes \Delta_A(aa')))\\
		&=&(1\otimes a''\otimes 1\otimes 1)(\imath_Y\otimes T\otimes\imath_A)((y''\otimes 1)\Delta_Y(yy')\otimes \Delta_A(aa'))\\
		&\stackrel{(\ast)}{=}&(1\otimes a''\otimes 1\otimes 1)(\imath_Y\otimes T\otimes\imath_A)(y''y_{(1)}\otimes y_{(2)}\otimes\Delta_A(a))(\imath_Y\otimes T\otimes\imath_A)(ey'_{(1)}\otimes y'_{(2)}\otimes \Delta_A(a'))\\		
		&\stackrel{(\ast\ast)}{=}&(1\otimes a''\otimes 1\otimes 1)(\imath_Y\otimes T\otimes\imath_A)(y''y_{(1)}\otimes y_{(2)}\otimes\Delta_A(a))(1\otimes c\otimes 1\otimes1)(\imath_Y\otimes T\otimes\imath_A)(ey'_{(1)}\otimes y'_{(2)}\otimes \Delta_A(a'))\\
		&=&(y''\otimes a''\otimes 1\otimes 1)\overline{\Delta}(y\otimes a)(e\otimes c\otimes1\otimes 1)\overline{\Delta}(y'\otimes a')\\
		&=&(y''\otimes a''\otimes 1\otimes 1)\overline{\Delta}(y\otimes a)\overline{\Delta}(y'\otimes a'),
	\end{eqnarray*}
	for all $y''\in Y$, $a''\in A$. The calculations for the second component are similar. Therefore, $\overline{\Delta}(yy'\otimes aa')=\overline{\Delta}(y\otimes a)\overline{\Delta}(y'\otimes a')$, for all $y,y'\in Y$, $a,a'\in A.$
		
	In $(\ast)$, we consider $e\in Y$ such that $y''y_{(1)}e=y''y_{(1)}$ and use that $T$ is a homomorphism in the extension.
	
In equality $(\ast\ast)$, we use $(y''\otimes a''\otimes 1\otimes 1)\overline{\Delta}(y\otimes a)\in Y\otimes A\otimes Y\otimes A$, then there exists an element $c\in A$ (local unit) on the second tensor. Consequently, we can apply the definition of $\overline{T}_2.$ 
	
	Finally, in the last equality we apply the definition to $(y''\otimes a''\otimes 1\otimes 1)\overline{\Delta}(y\otimes a)$ in $Y\otimes A\otimes Y\otimes A$ and use that $y''y_{(1)}e=y''y_{(1)}.$
\end{proof}

The multiplier bialgebra presented in Theorem \ref{thmultiplierbialg} is denoted by ${Y\rtimes A}.$

\begin{defi} Let $C$ be a multiplier bialgebra. We say that $C$ has \textit{left counit} (resp. \textit{right counit}) if there exists a homomorphism $\varepsilon:C\longmapsto{\Bbbk}$ such that $(\varepsilon\otimes\imath)\Delta=\imath$ (resp. $(\imath\otimes\varepsilon)\Delta=\imath$). Moreover, $C$ is \textit{counitary} if has left and right counit.
\end{defi}

\begin{pro}\label{Counidadeequerda}
	The multiplier bialgebra ${Y\rtimes A}$ has left counit given by
$\overline{\varepsilon}(y\otimes a)=\varepsilon_Y(y)\varepsilon_A(a),$ for all $y\in Y$ and $a\in A.$
\end{pro}
\begin{proof} It is straightforward to see that $	\overline{\varepsilon}$ is a homomorphism.
		\begin{eqnarray*}
		& \ &(\overline{\varepsilon}\otimes\imath_Y\otimes\imath_A)(\overline{\Delta}(y\otimes a)(1\otimes 1\otimes y'\otimes a'))\\
		&=&(\varepsilon_Y\otimes\varepsilon_A\otimes\imath_Y\otimes\imath_A)[((\imath_Y\otimes T\otimes\imath_A)(\Delta_Y(y)\otimes a_{(1)}\otimes a_{(2)}a'))(1\otimes1\otimes y'\otimes 1)]\\
		&=&(\varepsilon_A\otimes\imath_Y\otimes\imath_A)[(T\otimes\imath_A)(\varepsilon_Y\otimes\imath_Y\otimes\imath_A\otimes\imath_A)(\Delta_Y(y)\otimes a_{(1)}\otimes a_{(2)}a')(1\otimes y'\otimes 1)]\\
		&=&(\varepsilon_A\otimes\imath_Y\otimes\imath_A)[(T\otimes\imath_A)((\varepsilon_Y\otimes\imath_Y)\Delta_Y(y)\otimes a_{(1)}\otimes a_{(2)}a')(1\otimes y'\otimes 1)]\\
		&=&(\varepsilon_A\otimes\imath_Y\otimes\imath_A)[(T(y\otimes a_{(1)})\otimes a_{(2)}a')(1\otimes y'\otimes 1)]\\
		&=&((\varepsilon_A\otimes\imath_Y)(\rho(y)(a_{(1)}\otimes 1))\otimes a_{(2)}a')(y'\otimes 1)\\
		&=&(((\varepsilon_A\otimes\imath_Y)\rho(y))\varepsilon_A(a_{(1)})\otimes a_{(2)}a')(y'\otimes 1)\\
		&\stackrel{\ref{defcomodcopar}(i)}{=}&(y\otimes aa')(y'\otimes 1)\\
		&=&(y\otimes a)(y'\otimes a'),
	\end{eqnarray*}
	for all $y'\in Y$, $a'\in A.$ Then, $(\overline{\varepsilon}\otimes\imath_Y\otimes\imath_A)(\overline{\Delta}(y\otimes a))=y\otimes a,$ for all $y\in Y$, $a\in A.$
\end{proof}

Through a similar calculation to the one presented in the demonstration of Proposition \ref{Counidadeequerda}, it is possible to observe that in the partial case we do not guarantee $\overline{\varepsilon}$ as a right counit of ${Y\rtimes A}$. However, we can show the existence of a counit in a smaller subspace. For this purpose, we will present the following result.

\begin{pro}\label{paracoprosmashparc}Consider $C$ a multiplier  bialgebra with left counit and define the vector space generated 
	$$C_b:=\langle a_{(1)}\varepsilon(a_{(2)}b);\ \  \varepsilon(b)=1_{\Bbbk}\rangle_{\Bbbk}.$$
	Then:
	\begin{enumerate}
		\item [(i)] $C_b$ is a subalgebra of $C;$
		\item [(ii)] $C_b$ is a multiplier subbialgebra of $C;$
		\item [(iii)] $C_b$ has counit.
	\end{enumerate}
\end{pro}
\begin{proof}
	$(i)$ We initially observe that, for all $a\in C$, $(\imath\otimes\varepsilon)(\Delta(a)(1\otimes b))=(\imath\otimes\varepsilon)\Delta(a),$ as multipliers in $M(C).$
	Indeed, for every $c\in C$, we have
	\begin{eqnarray*}
		(\imath\otimes\varepsilon)(\Delta(a)(1\otimes b))(c)&=&a_{(1)}c\varepsilon(a_{(2)}b)\\
		&=&(\imath\otimes\varepsilon)((\Delta(a)(c\otimes 1))(1\otimes b))\\
		&=&(\imath\otimes\varepsilon)(\Delta(a)(c\otimes 1))\varepsilon(b)\\
		&=&(\imath\otimes\varepsilon)\Delta(a)(c).
	\end{eqnarray*}
	
Now we prove that $C_b$ is a subalgebra of $C.$ Given $a_{(1)}\varepsilon(a_{(2)}b)$, $d_{(1)}\varepsilon(d_{(2)}b)\in C_b$, 
	\begin{eqnarray*}
		(a_{(1)}\varepsilon(a_{(2)}b))(d_{(1)}\varepsilon(d_{(2)}b))&=&(\imath\otimes\varepsilon)(\Delta(a)(1\otimes b))(\imath\otimes\varepsilon)(\Delta(d)(1\otimes b))\\
		&=&(\imath\otimes\varepsilon)(\Delta(a))(\imath\otimes\varepsilon)(\Delta(d)(1\otimes b))\\
		&=&(\imath\otimes\varepsilon)(\Delta(ad)(1\otimes b))\\
		&=&(ad)_{(1)}\varepsilon((ad)_{(2)}b)\in C_b.
	\end{eqnarray*}
	$(ii)$ Note that $\Delta(C_b)(1\otimes C_b)\subseteq C_b\otimes C_b,$ 
	\begin{eqnarray*}
		\Delta(a_{(1)}\varepsilon(a_{(2)}b))(1\otimes d_{(1)}\varepsilon(d_{(2)}b))&=&\Delta(a_{(1)})(1\otimes d_{(1)})\varepsilon(a_{(2)}b)\varepsilon(d_{(2)}b)\\
		&=&(a_{(1)}\otimes a_{(2)}d_{(1)})\varepsilon(a_{(3)}b)\varepsilon(d_{(2)}b)\\
		&\stackrel{(\ast)}{=}& a_{(1)}\otimes \varepsilon(a_{(2)}b)a_{(3)}d_{(1)}\varepsilon(a_{(4)}b)\varepsilon(d_{(2)}b)\\
		&\stackrel{(i)}{=}&a_{(1)}\varepsilon(a_{(2)}b)\otimes (a_{(3)}d)_{(1)}\varepsilon((a_{(3)}d)_{(2)}b)\in C_b\otimes C_b.
	\end{eqnarray*}
In $(\ast)$ we used that $(\varepsilon\otimes\imath)(\Delta(a)(b\otimes 1))=(\varepsilon\otimes\imath)\Delta(a)=\imath(a),$
for all $a\in C$, \textit i.e. $\varepsilon$ is a left counit. Similarly, the subsets $(C_b\otimes 1)\Delta(C_b)$, $(1\otimes C_b)\Delta(C_b)$ and $\Delta(C_b)(C_b\otimes 1)$ lie in $C_b\otimes C_b.$ Thus, it is easy to see that $\Delta(C_b)\subseteq M(C_b\otimes C_b).$ 

%Indeed, previously we proved that $\Delta(C_b)(1\otimes C_b)\subseteq C_b\otimes C_b$ and $(C_b\otimes 1)\Delta(C_b)\subseteq C_b\otimes C_b.$ Thus we obtain  $\Delta(C_b)(C_b\otimes C_b)\subseteq C_b\otimes C_b$  and $(C_b\otimes C_b)\Delta(C_b)\subseteq C_b\otimes C_b$, since $C_b$ is subalgebra. Furthermore, for each $c\in C_b$, $\Delta(c)$ satisfies the compatibility condition in $M(C_b\otimes C_b)$ because it is already satisfied in $M(C\otimes C)$.	Therefore, $C_b$ is a multiplier subbialgebra of $C.$
	
	$(iii)$ Since $C_b$ is a multiplier subbialgebra of $C$,  $C_b$ has left counit. Moreover,	
	\begin{eqnarray*}
		(\imath\otimes\varepsilon)\Delta(a_{(1)}\varepsilon(a_{(2)}b))c&=&(\imath\otimes\varepsilon)(\Delta(a_{(1)}\varepsilon(a_{(2)}b))(c\otimes1))\\
		&=&a_{(1)}c\varepsilon(\varepsilon(a_{(2)})a_{(3)}b)\\
		&=&a_{(1)}\varepsilon(a_{(2)}b)c,
	\end{eqnarray*}
	for all $c\in C,$ where in the last equality we used the left counit of $C_b$. Therefore, $C_b$ has counit.
\end{proof}	
\begin{cor} Let ${Y\rtimes A}$ be the multiplier bialgebra. Then, $\overline{Y\rtimes A}=(\imath_Y\otimes\imath_A\otimes\overline{\varepsilon})\overline{\Delta}(Y\rtimes A)$ is a counitary multiplier bialgebra.
\end{cor}
\begin{proof} Straightforward from Propositions \ref{Counidadeequerda} and \ref{paracoprosmashparc}. 	
\end{proof}

 The counitary multiplier bialgebra 
	$\overline{Y\rtimes A}=(\imath_Y\otimes \imath_A\otimes\overline{\varepsilon})\overline{\Delta}({Y\rtimes A})$ is called the \textit{Partial Smash Coproduct}.

	\begin{pro}
		Let $Y$ be the symmetric partial $A$-comodule bialgebra via \begin{eqnarray*}
			\rho: Y&\longrightarrow & M(A\otimes Y)\\
			y &\longmapsto & h\otimes y
		\end{eqnarray*}	
		where $h\in M(A)$, $h\otimes h=(h\otimes 1)\Delta(h)=(1\otimes h)\Delta(h)$, $\varepsilon(h)=1_\Bbbk$ and $A$ is commutative. Then:
		\begin{enumerate}
			\item[(i)] $hA$ is a multiplier Hopf subalgebra of $A$;
			\item[(ii)] $\overline{Y\rtimes A}\simeq Y\otimes hA$ as counitary multiplier bialgebra.
			\end{enumerate}
	\end{pro}
\begin{proof}
	(i) We need to show that $(hA,\Delta_{hA})$ is a multiplier Hopf algebra, where $\Delta_{hA}$ denotes the coproduct $\Delta_A$ restrict to $hA$.
	
	Indeed, using that $A$ is a commutative nondegenerate algebra, $hA$ also is a commutative nondegenerate algebra. Moreover, for all $ha, hb\in hA$,
	\begin{eqnarray*}
	\Delta_{hA}(ha)(1\otimes hb)&=&h_{(1)}a_{(1)}\otimes h_{(2)}a_{(2)}hb\\
	&=& h_{(1)}a_{(1)}\otimes h_{(2)}ha_{(2)}b\\
	&=& ha_{(1)}\otimes hha_{(2)}b\\
	&=& ha_{(1)}\otimes ha_{(2)}b.
	\end{eqnarray*}
	 Similarly,  $\Delta_{hA}(hA)(hA\otimes 1), (1\otimes hA)\Delta_{hA}(hA)$ and $(hA\otimes 1)\Delta_{hA}(hA)$ lies in $hA\otimes hA$. Therefore, $hA$ is a multiplier Hopf subalgebra.
	 
	 (ii) By Corollary \ref{paracoprosmashparc}, the elements $y\otimes a$ are equal to $y\otimes ha$ in $\overline{Y\rtimes A}$, for all $y\in Y$, $a\in A$. Indeed,
	 \begin{eqnarray*}
	 (\imath_Y\otimes \imath_A\otimes\overline{\varepsilon})\overline{\Delta}({y\otimes a})&=& (\imath_Y\otimes \imath_A\otimes\overline{\varepsilon})(\overline{\Delta}(y\otimes a)(1\otimes1\otimes w\otimes b))\\
	 &=&(\imath_Y\otimes \imath_A\otimes\overline{\varepsilon})((\imath_Y\otimes T)(\Delta_Y(y)\otimes a_{(1)})(1\otimes1\otimes w)\otimes a_{(2)}b)\\
	 &=&y_{(1)}\otimes ha_{(1)}\varepsilon_Y(y_{(2)}w)\varepsilon_{A}(a_{(2)}b)\\
	 &=&y\otimes ha.
	 \end{eqnarray*}
	where $w\in Y$ and $b\in A$ such that $\varepsilon_Y(w)=\varepsilon_A(b)=1_{\Bbbk}$. Thus, $\overline{Y\rtimes A}=\Bbbk \langle y\otimes ha, \ y\in Y \ \mbox{and} \ a\in A \rangle$.
	
	Moreover, for all $y\otimes ha, y'\otimes ha'\in \overline{Y\rtimes A}$, $(y\otimes ha)(y'\otimes ha')=(yy'\otimes haha')$ and
	\begin{eqnarray*}
	\overline{\Delta}(y\otimes ha)((y'\otimes a')\otimes (y''\otimes a''))&=& \overline{T}_1((y\otimes ha)\otimes (y''\otimes a''))((y'\otimes a')\otimes(1\otimes1))\\
	&=& ((\imath_Y\otimes T)(\Delta_Y(y)\otimes h_{(1)}a_{(1)})(1^2\otimes y'')\otimes h_{(2)}a_{(2)}a'')((y'\otimes a')\otimes(1\otimes1))\\
	&=& ((\imath_Y\otimes T)(\Delta_Y(y)(y'\otimes 1)\otimes h_{(1)}a_{(1)})(1^2\otimes y'')\otimes h_{(2)}a_{(2)}a'')(1\otimes a'\otimes 1^2)\\
	&=& (y_{(1)}y'\otimes T(y_{(2)}\otimes h_{(1)}a_{(1)})(1\otimes 1\otimes y'')\otimes h_{(2)}a_{(2)}a'')(1\otimes a'\otimes 1\otimes1)\\
	&=& (y_{(1)}y'\otimes hh_{(1)}a_{(1)} \otimes y_{(2)}y''\otimes h_{(2)}a_{(2)}a'')(1\otimes a'\otimes 1\otimes1)\\
	&=& (y_{(1)}y'\otimes ha_{(1)} \otimes y_{(2)}y''\otimes ha_{(2)}a'')(1\otimes a'\otimes 1\otimes1)\\
	&=&(\imath_Y\otimes\tau_{Y\otimes A}\otimes\imath_A)(\Delta_Y(y)\otimes (h\otimes h)\Delta_A(a))((y'\otimes a')\otimes (y''\otimes a'')), 
	\end{eqnarray*}
for all $y'\otimes a', y''\otimes a''\in\overline{Y\rtimes A}$. Then, 
\begin{eqnarray*}
\overline{\Delta}(y\otimes ha)&=&(\imath_Y\otimes\tau_{Y\otimes A}\otimes\imath_A)(\Delta_Y(y)\otimes (h\otimes h)\Delta_A(a))\\
&=&(\imath_Y\otimes\tau_{Y\otimes hA}\otimes\imath_{hA})(\Delta_Y(y)\otimes \Delta_{hA}(ha))\\
&=&\Delta(y\otimes ha).
\end{eqnarray*}
And, finally $\overline{\varepsilon}(y\otimes ha)=\varepsilon(y\otimes ha)$. The maps $\Delta$ and $\varepsilon$ are the usual coproduct and counit on $Y\otimes hA,$ respectively. Therefore, $\overline{Y\rtimes A}$ and $Y\otimes hA$ have the same structure of counitary multipler bialgebra.
\end{proof}
	\begin{exa} 
				Let $Y$ be the symmetric partial $A_G$-comodule coalgebra via  $\rho(y)=h\otimes y$, given by Example \ref{exemcomoparcviah}. It is easy to check that $\rho$ also defines a structure of symmetric partial $A_G$-comodule bialgebra on $Y$. Thus 
		$\overline{Y\rtimes A_G}\simeq Y\otimes fA_G$ as counitary multiplier bialgebra. Moreover, $ Y\otimes fA_G\simeq Y\otimes A_N$ as counitary multiplier bialgebra, since
		$$f\delta_g(s)=f(s)\delta_g(s)=\left\{
		\begin{array}{rl}
		\delta_g(s) & \text{, } s\in N\\
		0 & \text{, otherwise }
		\end{array},\right.$$	
	for all $f\delta_g\in fA_{G}$ and $s\in G.$ 	
	\end{exa}

\section{Acknowledgments}
The authors would like to thank A. Van Daele for his precious suggestions and disposition to enlight us about the theory of Multiplier
Hopf Algebras. 

\bibliographystyle{abbrv}

\end{document}